\documentclass[12pt]{amsart}

\usepackage{cmap}
\usepackage{amsfonts,amssymb,amscd,amsmath,latexsym,amsbsy,tikz-cd,hyperref}
\numberwithin{equation}{section}
\hypersetup{hidelinks}
\newtheorem{theorem}{Theorem}[section]
\newtheorem{lemma}[theorem]{Lemma}
\newtheorem{proposition}[theorem]{Proposition}
\newtheorem{corollary}[theorem]{Corollary}
\theoremstyle{definition}
\newtheorem{definition}[theorem]{Definition}

\newtheorem{example}[theorem]{Example}

\newtheorem{remark}[theorem]{Remark}

\newcommand{\id}{\text{id}}

\newcommand{\Hom}{\text{Hom}}

\newcommand{\Id}{\text{Id}}
\newcommand{\Aut}{\text{Aut}}

\newcommand{\Rep}{\text{Rep}}

\newcommand{\C}{\mathcal{C}}

\newcommand{\ben}{\begin{enumerate}}
\newcommand{\een}{\end{enumerate}}

\theoremstyle{plain}

\newtheorem*{sol}{Solution}

\theoremstyle{definition}

\theoremstyle{remark}

\newcommand{\solu}[1]{\begin{sol}{\bf (\ref{#1})}}

\def\C{\mathcal{C}}
\def\D{\mathcal{D}}
\def\E{\mathcal{E}}

\def\Aut{\mathop{\mathrm{Aut}}\nolimits}

\def\Z{\mathcal{Z}}

\def\Hom{\mathrm{Hom}}

\def\Im{\mathop{Im}}

\def\B{\mathcal{B}}

\def\Vec{\mathrm{Vec}}

\def\k{\mathbf{k}}
\def\ev{\mathrm{ev}}
\def\coev{\mathrm{coev}}
\def\Rep{\mathop{\mathrm{Rep}}\nolimits}

\begin{document}

\title{Twisted Deligne products of semisimple tensor categories}
\author{Pavel Etingof}
\address{Department of Mathematics
Massachusetts Institute of Technology
\newline
77 Massachusetts Avenue,
Cambridge, MA 02139,
USA
}
\email{etingof@math.mit.edu}

\author{Dmitri Nikshych}
\address{Department of Mathematics
University of New Hampshire
\newline Durham, NH 03824}
\email{ Dmitri.Nikshych@unh.edu}

\author{Victor Ostrik}
\address{Department of Mathematics,
University of Oregon
\newline
Eugene, OR 97403
}
\email{vostrik@uoregon.edu}

\begin{abstract}
 We discuss the classification of twisted Deligne products of two semisimple tensor categories $\C,\D$, i.e., categorifications of the tensor product of their Grothendieck rings in which the factors
are categorified by $\C$ and $\D$. In particular, we show that if both factors have trivial universal grading groups, or if one factor has trivial universal grading group and every tensor structure on its identity functor is isomorphic to the standard one, then the only twisted Deligne product is the ordinary one. Together with the work of \cite{MPP}, this reduces the classification of exact factorizations of fusion categories to successive graded extensions with compatible embeddings of the factors. In the Appendix we introduce the notion of categorical $n$-cocycles for $n=2,3,4$ and show that they are all pullbacks of group $n$-cocycles from the universal grading group of the underlying based ring. In the case of $4$-cocycles, this answers the rigid case of a question of Johnson-Freyd, Ostrik and Yu.
\end{abstract}

\maketitle

\centerline{\bf To George Lusztig for his 80th birthday with admiration}

\tableofcontents

\section{Introduction}

Let $\mathbf{k}$ be an algebraically closed field.
Let $\C,\D,\E$ be semisimple tensor categories over $\mathbf{k}$ (\cite{EGNO}, Definition 4.1.1).
Following \cite{G},
we say that $\E=\C\bullet \D$ is an {\it exact factorization} if $\E$ is equipped with tensor subcategory embeddings $\C,\D\hookrightarrow \E$ such that the tensor product defines an abelian equivalence $\C\boxtimes \D\to \E$. Equivalently, $\C\D=\E$ and $\C\cap\D=\Vec$ (cf.~\cite{G} for fusion categories). The same criterion applies with infinitely many simples: if $\C\cap\D=\Vec$, then only the unit contributes to $\Hom(X'^{*}\otimes X,Y'\otimes Y^*)\cong\Hom(X\otimes Y,X'\otimes Y')$ for simple $X,X'\in\C$, $Y,Y'\in\D$.

Furthermore, we say that $\E=\C\bullet \D$ is a {\it twisted Deligne product} if
$$
X\otimes Y\cong Y\otimes X,\ X\in \C,\ Y\in \D.
$$
That is, a twisted Deligne product of $\C,\D$ is a categorification of the tensor product ${\rm Gr}(\C)\otimes {\rm Gr}(\D)$ of the Grothendieck rings of $\C,\D$ in which the subcategories
corresponding to the factors are identified with $\C,\D$, respectively. The simplest example of a twisted Deligne product is the usual Deligne product, $\E=\C\boxtimes \D$.

\begin{example}\label{groupprod} For a group $L$ and $\omega\in H^3(L,\mathbf{k}^\times)$, denote by $\Vec_L^\omega$
the tensor category of finite-dimensional $L$-graded vector spaces with associativity defined by a cocycle representing $\omega$. Let $L=GH$ be an exact factorization of groups, and $\omega\in H^3(L,\mathbf{k}^\times)$.
Let $\omega_G,\omega_H$ be the restrictions of $\omega$ to $G,H$. Then $\Vec_L^\omega=\Vec_G^{\omega_G}\bullet \Vec_H^{\omega_H}$ is an exact factorization. Moreover, this is a twisted Deligne product iff $G,H$ commute, i.e., $L=G\times H$.
\end{example}

The main goal of this paper is to discuss classification of twisted Deligne products up to tensor equivalences intertwining the given embeddings of $\C$ and $\D$ up to monoidal isomorphism. We call such equivalences {\it marked}. Namely, one of our main results
is the following theorem.

For a tensor category $\mathcal C$, let $U_{\mathcal C}$ be its universal grading group
(\cite[3.2]{GN}, \cite[Subsection 4.14]{EGNO}), and let ${\rm Aut}_1(\C)$ be the group of isomorphism classes of tensor structures on the identity functor ${\rm Id}_\C$.

 \begin{theorem} \label{maint} (i) (Theorem \ref{part1}) If $U_\C= U_\D=1$ then
 any twisted Deligne product $\C\bullet \D$ is equivalent to the usual Deligne product $\C\boxtimes \D$
 (tautologically on $\C$ and $\D$).

 (ii) (Theorem \ref{part2})
 The same holds if $U_\C=1$ and $\Aut_1(\C)=1$ or $U_\D=1$ and $\Aut_1(\D)=1$.
 \end{theorem}

 As indicated in \cite{MPP}, Remark 3.27, a classification of twisted Deligne products of fusion categories in principle allows one to classify all their exact factorizations in group-theoretical terms using the extension theory of \cite[Theorem~7.7]{ENO} (see \cite[Theorem~8.5]{DN} in arbitrary characteristic). In fact, for this one doesn't even need a full classification of twisted Deligne products: Theorem \ref{maint}(i) is already sufficient. Namely, let $\C_{\rm ad}$ denote the adjoint subcategory of $\C$ (\cite{EGNO}, Subsection 4.14). In Corollary 3.24 and Proposition 3.25 of \cite{MPP} it is shown that
if $\E=\C\bullet \D$ is an exact factorization of fusion categories then $\E_{\rm ad}=\C_{\rm ad}\bullet \D_{\rm ad}$ is a twisted Deligne product. The category $\E$ can then be recovered from $\E_{\rm ad}$ as an extension of $\E_{\rm ad}$ by the universal grading group $U_\mathcal E$, and such extensions can in principle be classified as in \cite{ENO} if one knows the Brauer--Picard $3$-groupoid of $\E_{\rm ad}$.

Iterate the passage to adjoint subcategories. By \cite[Proposition~3.25]{MPP}, the group $U_\E$ factors exactly as $U_\C\bullet U_\D$. Thus, as long as either factor has nontrivial universal grading group, passage from $\E$ to $\E_{\rm ad}$ divides the Frobenius--Perron dimension by $|U_\E|\ge 2$. The process terminates with both universal grading groups trivial, where Theorem~\ref{maint}(i) gives the ordinary Deligne product. To reverse a step, one chooses a graded extension of $\E_{\rm ad}$ and extends the two given embeddings to tensor embeddings of $\C$ and $\D$ whose images factor the extension exactly. For the marked classification, equivalences of these data must intertwine both embeddings.

Admittedly, this classification is a bit implicit. After all, computation of Brauer-Picard groups of Deligne products is a rather nontrivial problem, as is classifying group extensions of a given category, and the answer can be complicated. However, in some cases the answer is very simple. For example, Theorem \ref{maint}(i)
and the results of \cite{MPP} imply that if $\C,\D$ are fusion categories with $U_\C=U_\D=1$ then every exact factorization
$\C\bullet \D$ is the ordinary Deligne product $\C\boxtimes \D$.

The paper is organized as follows. Section 2 contains preliminaries. In Section 3 we give a partial classification of twisted Deligne products. Finally, in the Appendix we introduce the notion of categorical $n$-cocycles for $n=2,3,4$ and show that they are all pullbacks of group $n$-cocycles on the universal grading group of the underlying based ring. In the case of $4$-cocycles, this answers the rigid case of a question of Johnson-Freyd, Ostrik and Yu. In general the answer is negative, as shown by Example~\ref{nonrigid-defect}.

{\bf AI use.}  The first version of this paper containing Theorem \ref{maint} was written in 2017. 
Then the paper was finalized in June-July 2026 with the help of ChatGPT 5.5 Pro. Namely, ChatGPT was used to write the new section on categorical 4-cocycles and to complete and edit proofs. In preparing the present revision, we used ChatGPT 6.2 Pro to compare the referee reports with the paper, draft revisions of the proofs, and proofread the manuscript. 

{\bf Acknowledgements.} We are grateful to the referees for many useful comments. The work of P.E. was supported by the NSF grants DMS-2001318 and DMS-2502467. The work of D.N.\ was supported  by  the  National  Science  Foundation  under  Grant No.\ DMS-2302267.

\section{Preliminaries}

\subsection{The universal grading group} Throughout, let $R$ be a unital based ring in the sense of \cite[Subsection 3.1]{EGNO} with distinguished basis $I$ (containing the unit). Then $R$ carries a  partial order given by $a\le b$ if $b-a$ has non-negative coefficients.

Introduce a relation on $I$ by setting $x\sim y$ if there exist $u_1,...,u_n\in I$ such that $x,y\le u_1...u_n$.
This is an equivalence relation: if $x\sim y$ and $y\sim z$
then $x,y\le u_1...u_n$ and $y,z\le v_1...v_m$, hence
$$
x\le xy^*y\le u_1...u_ny^*v_1...v_m,\ z\le yy^*z\le u_1...u_ny^*v_1...v_m.
$$
Let $U(R):=I/\sim$ and denote the equivalence class of $x\in I$ by $\deg x$. This will be the degree of $x$ in the universal grading.
Thus $I=\sqcup_{g\in U(R)}I_g$, where $I_g$ is the equivalence class in $I$ corresponding to $g$. If $x,y,z\in I$
and $x\sim y$, then for any $v,w\in I$ with $v\le xz$, $w\le yz$ or $v\le zx$, $w\le zy$, we have $v\sim w$.
Thus we have a well-defined associative operation on $U(R)$ given by $\deg x\cdot \deg y=\deg z$ if $z\le xy$. Moreover, $1=\deg\mathbf{1}$ is a unit for this operation. Finally,
if $x\sim y$ then $x^*\sim y^*$, so for all $g\in U(R)$ there is a well-defined
inverse $g^{-1}=g^*$. Thus $U(R)$ is a group. It is called the {\it universal grading group} of $R$
 (\cite[3.2]{GN}, \cite[Subsection 4.14]{EGNO}). Thus the based ring $R$ has a $U(R)$-grading: $R=\oplus_{g\in U(R)}R_g$, where $R_g:=\mathbb{Z}I_g$.

The group $U(R)$ has the following universal property.

\begin{lemma}\label{homom}  Let $G$ be a group and $f : I\to G$ a function
such that $f(y)=f(x)f(z)$ whenever $y\le xz$. Then there is a unique homomorphism $\overline{f}: U(R)\to G$ such that $f(x)=\overline{f}(\deg x)$.
\end{lemma}

\begin{proof} Let $x,y\le u_1\cdots u_n$. We prove that $f(x)=f(y)$ by induction in $n$.
The base $n=1$ is trivial. Let $n>1$. Then there exist $x',y'\le u_1...u_{n-1}$ such that $x\le x'u_n,y\le y'u_n$.
Thus $f(x')=f(y')$ by the induction assumption, so by the hypothesis $f(x)=f(x')f(u_n)=f(y')f(u_n)=f(y)$. Hence $f$ is constant on $\sim$-classes and descends uniquely to $U(R)$.
\end{proof}

In particular, faithful $G$-gradings on $R$ correspond to surjective homomorphisms
$U(R)\to G$.

If $\C$ is a semisimple tensor category and $R={\rm Gr}(\C)$ is the Grothendieck ring of $\C$,
then we will call $U(R)$ the {\it universal grading group of $\C$} and denote it by $U_\C$.
In this case we have $\C=\oplus_{g\in U_\C}\C_g$, where $\C_g$ is spanned by simple objects $X$ with
$\deg X=g$. Thus faithful $G$-gradings on $\C$ correspond to surjective homomorphisms
$U_\C\to G$.

\subsection{The group \texorpdfstring{${\rm Aut}_1(\C)$}{Aut1(C)}}
Let $\C$ be a semisimple tensor category over $\mathbf{k}$ with the set of simple objects ${\rm Irr}(\C)$. Let
$\Aut(\C)$ be the group of tensor autoequivalences of $\C$ up to isomorphism. For a tensor autoequivalence $\phi:\C\to \C$, denote its class in $\Aut(\C)$ by $\widehat\phi$. Let $\Aut_1(\C)\subset \Aut(\C)$ be the subgroup of autoequivalences acting trivially on ${\rm Gr}(\C)$, i.e.,
tensor structures on the identity functor of $\C$ (after choosing isomorphisms $X\cong \phi(X)$ for simple $X\in \C$).

\begin{remark}\label{softaut}
The group $\Aut_1(\C)$ has appeared in the literature under several names. It was studied categorically by Davydov \cite{D1,D2} as the group of twisted forms, or tensor deformations, of the identity functor. In the Hopf algebra setting, the corresponding cohomology group was introduced by Schauenburg \cite{S}; it was subsequently called the second \emph{lazy cohomology} group and systematically studied by Bichon and Carnovale \cite{BC}. More precisely, for a finite-dimensional semisimple Hopf algebra $H$, the group $\Aut_1(\Rep(H))$ is identified with the group of invariant twists on $H$ modulo gauge transformations by $u\in Z(H)^\times$ with $\varepsilon(u)=1$, or, equivalently, with the lazy cohomology group of $H^*$. Davydov later called its elements \emph{soft} tensor autoequivalences \cite{D3}. Guillot and Kassel \cite{GK} showed that this group need not be abelian and gave such examples already for $\C=\Rep_{\mathbb C}(G)$ for finite groups $G$ (\cite{GK}, Proposition 1.4).
\end{remark}

We have a natural homomorphism
$$
\theta: H^2(U_\C,\mathbf{k}^\times)\to \Aut_1(\C)
$$
given by $\theta([\beta])|_{W_1\otimes W_2}=\beta(\deg W_1,\deg W_2)$, where $\beta$ is a normalized 2-cocycle on $U_\C$.

\begin{lemma}\label{inje}  $\theta$ is injective. Thus it identifies $H^2(U_\C,\mathbf{k}^\times)$
with a central subgroup of $\Aut_1(\C)$.
\end{lemma}

\begin{proof} If $\theta([\beta])=1$ then $\beta(\deg W_1,\deg W_2)=f(W_1)f(W_2)f(W_3)^{-1}$ for each $W_3\subset W_1\otimes W_2$, where $f: {\rm Irr}(\C)\to \mathbf{k}^\times$. Thus, if $X$ occurs in $W_1\otimes\cdots\otimes W_n$ then
$f(X)^{-1}\prod_i f(W_i)$ depends only on the degrees $\deg W_i$ of the tensor factors. So if $X,X'\subset W_1\otimes\cdots\otimes W_n$ then $f(X)=f(X')$. Hence $f(X)=\overline{f}(\deg X)$, where $\overline{f}$ is a function on $U_\C$, and $\beta=d\overline{f}$. Also, the image of $\theta$ acts on each simple summand of $X\otimes Y$ by the scalar $\beta(\deg X,\deg Y)$, so it commutes with any tensor structure on the identity functor, implying that ${\rm Im}\theta$ is central.
\end{proof}

However, $\theta$ need not be surjective since according to Remark \ref{softaut}, $\Aut_1(\C)$ can be non-abelian.

\subsection{Twisted Drinfeld center}

Recall that the {\it Drinfeld center} of a tensor category $\C$ may be defined as
the category of $\C$-bimodule functors $\C\to \C$. More generally,
let $\phi: \C\to \C$ be a tensor autoequivalence. Then
we can define the $\C$-bimodule category $\C(\phi)$
which is $\C$ as a left $\C$-module category, while the right $\C$-action is twisted by $\phi$. That is,
for $Y\in \C(\phi)$ and $X\in \C$, $Y\circ X:=Y\otimes \phi(X)$.
Now define the {\it twisted Drinfeld center} of $\C$ to be the
category of $\C$-bimodule functors $\C\to \C(\phi)$. We denote it by $\Z_\phi(\C)$. If $\psi$ is another tensor autoequivalence, the category of $\C$-bimodule functors $\C(\psi)\to\C(\phi)$ is equivalent to $\Z_{\phi\psi^{-1}}(\C)$.

Recall that any module endofunctor of $\C$ as a left $\C$-module is right multiplication
by an object $Z$ of $\C$. Thus in concrete terms, an object
of $\Z_\phi(\C)$ can be realized as an object $Z$ of $\C$ together with a functorial isomorphism
$c_X: X\otimes Z\to Z\otimes \phi(X)$, $X\in \C$, satisfying the consistency relation
\begin{equation}\label{conrel}
c_{X\otimes Y}=(\id_Z\otimes \mathcal J_{X,Y})(c_X\otimes\id_{\phi(Y)})(\id_X\otimes c_Y),
\end{equation}
where $\mathcal J_{X,Y}:\phi(X)\otimes\phi(Y)\to \phi(X\otimes Y)$ is the tensor structure of $\phi$ and we suppress associativity isomorphisms (in particular, $c_{\mathbf{1}}=\id_Z$). Namely, a functor $F\in \Z_\phi(\C)$ gives rise to the object $Z_F:= F(\mathbf{1})$, and conversely, $F_Z(X)=X\otimes Z$.

Let $G$ be a group and $\beta: G\times G\to \k^\times$ be a 2-cocycle. Denote by
$\Rep_\beta(G)$ the category of finite-dimensional projective representations $V$
of $G$ with $2$-cocycle $\beta$, i.e. maps $\rho: G\to GL(V)$
such that
$$
\rho(gh)=\rho(g)\rho(h)\beta(g,h),\ g,h\in G.
$$

Let $\phi$ be a tensor autoequivalence of the semisimple tensor category $\C$, and write $\widehat\phi$ for its isomorphism class. Let $$\theta: H^2(U_\C,\k^\times)\to {\rm Aut}_1(\C)$$ be the map of Lemma \ref{inje}.

\begin{lemma}\label{twistedcenter} There exists a nonzero object of $\Z_\phi(\C)$ whose underlying object is a finite direct sum of copies of $\mathbf{1}$ if and only if
\[
\widehat{\phi}=\theta([\beta])\in \Aut_1(\C)\subset \Aut(\C)
\]
for some normalized $2$-cocycle $\beta$ on $U_\C$ with values in $\k^\times$ such that $\Rep_\beta(U_\C)\ne 0$. In this case $\widehat\phi$ determines $[\beta]$ uniquely.

For such a cocycle $\beta$, let $J^\beta$ be the tensor structure on $\Id_\C$ given on homogeneous objects $X,Y$ by
\[
J^\beta_{X,Y}=\beta(\deg X,\deg Y)\id_{X\otimes Y}.
\]
A choice of monoidal isomorphism $\eta:\phi\xrightarrow{\sim}(\Id_\C,J^\beta)$ gives an equivalence from $\Rep_\beta(U_\C)$ to the full subcategory of $\Z_\phi(\C)$ on these multiples of $\mathbf{1}$. It sends $(V,\rho)$ to $V\otimes\mathbf{1}$ with half-braiding
\[
c_X=\rho(\deg X)\otimes\eta_X^{-1}:V\otimes X\longrightarrow V\otimes\phi(X)
\]
for homogeneous $X$.
\end{lemma}

\begin{proof} Suppose $Z=V\otimes\mathbf{1}$ admits a half-braiding $c$ and $V\ne 0$. For each simple $X\in\C$, the isomorphism $c_X$ identifies $V\otimes X$ with $V\otimes\phi(X)$. Hence $\phi(X)\cong X$. Choose such identifications and write $\mathcal J$ for the induced tensor structure on $\Id_\C$. In these coordinates $c_X=\zeta_X\otimes\id_X$ for some $\zeta_X\in GL(V)$.
For a simple summand $W_3\subset W_1\otimes W_2$, \eqref{conrel} gives
\begin{equation}\label{concon}
\zeta_{W_3}\otimes \id=\zeta_{W_1}\zeta_{W_2}\otimes \gamma(W_1,W_2;W_3)
\end{equation}
on $V\otimes \Hom(W_3,W_1\otimes W_2)$, where $\gamma(W_1,W_2;W_3)$ is the operator by which $\mathcal J_{W_1,W_2}$ acts on $\Hom(W_3,W_1\otimes W_2)$.  It follows that $\gamma(W_1,W_2;W_3)$ is a scalar.

For $g\in GL(V)$, let $[g]$ be the image of $g$ in $PGL(V)$. Since $\gamma$ is a scalar, we have
\[
[\zeta_{W_3}]=[\zeta_{W_1}][\zeta_{W_2}]
\]
whenever $W_3\subset W_1\otimes W_2$. By Lemma~\ref{homom}, it follows that the assignment $X\mapsto[\zeta_X]$ factors through a homomorphism $U_\C\to PGL(V)$.  Choose lifts $\rho(g)\in GL(V)$ with $\rho(1)=\id_V$ and define a normalized $2$-cocycle $\beta$ on $U_\C$ by
\[
\rho(gh)=\rho(g)\rho(h)\beta(g,h),\ g,h\in U_\C.
\]

After rescaling the chosen identifications $\phi(X)\cong X$ by some scalars $\lambda_X\in \k^\times$, we may assume that $\zeta_X=\rho(\deg X)$.
Then \eqref{concon} implies
\[
\gamma(W_1,W_2;W_3)=\beta(\deg W_1,\deg W_2)
\]
whenever $W_3\subset W_1\otimes W_2$. Hence $\widehat\phi=\theta([\beta])$, and injectivity of $\theta$ gives uniqueness of $[\beta]$.

Conversely, fix $\beta$ and $\eta$ as in the statement, and use $\eta$ to identify $\phi$ with $(\Id_\C,J^\beta)$. For $Z=V\otimes\mathbf{1}$, the half-braiding in the statement then acts on $V\otimes X$ as $\rho(\deg X)\otimes\id_X$. The relation $\rho(gh)=\rho(g)\rho(h)\beta(g,h)$ is precisely \eqref{conrel} for homogeneous inputs of degrees $g,h$.

To see that every object of the subcategory arises in this way, let $c$ be any half-braiding on $V\otimes\mathbf{1}$ and write $\zeta_X$ for its action on $V$ when $X$ is simple. Iterating \eqref{concon} on a tensor word $W_1\cdots W_n$ expresses $\zeta_X$, for every simple constituent $X$, as $\zeta_{W_1}\cdots\zeta_{W_n}$ times a scalar determined by $\beta$ and the degrees of the $W_i$. Thus $\zeta_X$ depends only on $\deg X$, and \eqref{concon} gives the required projective representation. Morphisms in the twisted center are exactly the intertwiners.
\end{proof}

\subsection{The K\"unneth theorem}

Let $G,H$ be groups and let $A$ be an abelian group, regarded as a trivial $G\times H$-module.  For $n>0$, define restriction maps
$$
\xi_{0,n}: H^n(G\times H,A)\to H^n(G,A),\ \xi_{n,0}: H^n(G\times H,A)\to H^n(H,A).
$$
Let
\[
\Gamma_n(G,H,A):=\ker(\xi_{0,n}\oplus \xi_{n,0}).
\]
The two projections $G\times H\to G,H$ give canonical sections of $\xi_{0,n},\xi_{n,0}$, respectively,
 so for $n>0$ one has a canonical decomposition
\[
H^n(G\times H,A)\cong H^n(G,A)\oplus H^n(H,A)\oplus \Gamma_n(G,H,A).
\]

Write $C_n(L)$ for the normalized integral bar chains of $BL$ (trivial coefficients). The shuffle map from $C_*(H)\otimes C_*(G)$ to $C_*(G\times H)$ interleaves the lists $(1,h_i)$ and $(g_j,1)$, with the signs of the shuffles. Its chain-homotopy inverse is the Alexander--Whitney map, with the factors ordered $H,G$. Precomposition with the shuffle therefore takes a product-group cochain to the double complex
\[
 D^{p,q}=\Hom_{\mathbb Z}\bigl(C_p(H),\Hom_{\mathbb Z}(C_q(G),A)\bigr),
 \qquad d=d_H+(-1)^p d_G
\]
in additive notation. In multiplicative notation a negative shuffle sign means inversion.

Filtering by $p$ gives $E_2^{p,q}=H^p(H,H^q(G,A))$. To see that this sequence degenerates, put $K^*=\Hom(C_*(G),A)$. For each $q$, choose a free abelian presentation $0\to P_1\to P_0\to H^q(K)\to0$, with $P_1$ free. Lift a basis of $P_0$ to cocycles and lift the images of a basis of $P_1$ to primitives in $K^{q-1}$. The resulting two-term complexes, placed in degrees $q-1,q$, give, on taking their direct sum, a zigzag of quasi-isomorphisms between $K^*$ and its cohomology with zero differential. All $H$-actions here are trivial. Applying $\Hom(C_p(H),-)$ preserves these quasi-isomorphisms since $C_p(H)$ is free. The resulting comparison on column filtrations is an isomorphism from $E_1$ on, and the complex with zero vertical differential has no differentials past $E_2$. The two axes $p=0$ and $q=0$ split off in positive total degree since the bar boundaries in degree one vanish. Hence the natural filtration of $\Gamma_n$ has mixed graded pieces $H^p(H,H^{n-p}(G,A))$, $1\le p<n$. 

\begin{example}\label{ex1} Recall that if a group $L$ acts trivially on an abelian group $M$ then $H^1(L,M)=\Hom(L,M)$.
Therefore the groups $\Gamma_2$ and $\Gamma_3$ have the following structure.

1. One has \scriptsize
$$
\Gamma_2(G,H,A)=\Hom(H,\Hom(G,A))=\Hom(G,\Hom(H,A))=\Hom(G_{\rm ab}\otimes H_{\rm ab},A),
$$ \normalsize
where the subscript denotes abelianization.

2. There are natural short exact sequences
\[
0\to H^2(H,\Hom(G,A))\to \Gamma_3(G,H,A)\xrightarrow{\tau_H} \Hom(H,H^2(G,A))\to 0,
\]
and, after interchanging $G$ and $H$,
\[
0\to H^2(G,\Hom(H,A))\to \Gamma_3(G,H,A)\xrightarrow{\tau_G} \Hom(G,H^2(H,A))\to 0.
\]
The kernel of
\[
\tau_G\oplus\tau_H:\Gamma_3(G,H,A)\to \Hom(G,H^2(H,A))\oplus \Hom(H,H^2(G,A))
\]
may be described as the kernel of the map
\[
\eta: H^2(G,\Hom(H,A))\to \Hom(H,H^2(G,A))
\]
obtained by restricting $\tau_H$ to $\ker\tau_G=H^2(G,\Hom(H,A))$.  On cocycles, this map is represented by transposing the variables:
\[
\eta(\gamma)(h)(g_1,g_2)=\gamma(g_1,g_2)(h).
\]
For the crossing conventions in Section~3 we use $\tau_1=\tau_H$ and take $\tau_2$ to be $\tau_G$ followed by pointwise inversion, writing $A$ multiplicatively. Explicitly, let $\gamma\in Z^2(H,\Hom(G,A))$ be normalized and put
\[
 \omega((g_1,h_1),(g_2,h_2),(g_3,h_3))=\gamma(h_1,h_2)(g_3).
\]
The shuffle formulas give $\tau_1([\omega])=1$ and $\tau_2([\omega])(g)=[\gamma(-,-)(g)]^{-1}$.
\end{example}

\section{Towards classification of twisted Deligne products}

\subsection{Data attached to a twisted Deligne product}

Let $\E=\C\bullet\D$ be a twisted Deligne product. For simple $Y\in\D$, the functor $X\mapsto X\otimes Y$ identifies $\C$ with $\C_Y:=\C\otimes Y$ as a left $\C$-module category. Under this identification, the right action is twisted by a tensor autoequivalence $\phi_Y$ of $\C$. Since $Y\otimes X\cong X\otimes Y$, its class $\widehat\phi_Y$ lies in $\Aut_1(\C)$. Similarly, the left action of $\D$ on $X\otimes\D$, for simple $X\in\C$, gives a tensor autoequivalence $\psi_X$ with class $\widehat\psi_X\in\Aut_1(\D)$.

Choose crossings $\xi_{YX}:Y\otimes X\to X\otimes Y$ for simple $X,Y$, using the canonical maps when either object is the unit. Extend them by additivity. These crossings identify the underlying additive functors of $\phi_Y$ and $\psi_X$ with the identity functors of the respective categories.

Write $J^\phi=\mathcal J(\phi_Y)_{W_1,W_2}$ and $J^\psi=\mathcal J(\psi_X)_{V_1,V_2}$ for their tensor structures. With associators and identity factors suppressed, the crossing identities are
\begin{equation}\label{xiiden}
\begin{aligned}
\xi_{Y,W_1\otimes W_2}&=J^\phi\xi_{YW_2}\xi_{YW_1},\\
\xi_{V_1\otimes V_2,X}&=\xi_{V_1X}\xi_{V_2X}J^\psi.
\end{aligned}
\end{equation}
The following diagrams give \eqref{xiiden} with all associators included. We write tensor products by juxtaposition.
\[
\begin{tikzcd}[column sep=8em,row sep=2.1em]
 Y(W_1W_2) \arrow[r,"{\xi_{Y,W_1W_2}}"] \arrow[d,"{\alpha^{-1}}"']
   & (W_1W_2)Y \\
 (YW_1)W_2 \arrow[d,"{\xi_{Y,W_1}\otimes\id}"]
   & W_1(W_2Y) \arrow[u,"{(J^\phi\otimes\id)\alpha^{-1}}"'] \\
 (W_1Y)W_2 \arrow[r,"{\alpha}"']
   & W_1(YW_2) \arrow[u,"{\id\otimes\xi_{Y,W_2}}"']
\end{tikzcd}
\]
\[
\begin{tikzcd}[column sep=8em,row sep=2.1em]
 (V_1V_2)X \arrow[r,"{\xi_{V_1V_2,X}}"] \arrow[d,"{\alpha(J^\psi\otimes\id)}"]
   & X(V_1V_2) \\
 V_1(V_2X) \arrow[d,"{\id\otimes\xi_{V_2,X}}"]
   & (XV_1)V_2 \arrow[u,"{\alpha}"'] \\
 V_1(XV_2) \arrow[r,"{\alpha^{-1}}"']
   & (V_1X)V_2 \arrow[u,"{\xi_{V_1,X}\otimes\id}"']
\end{tikzcd}
\]

This shows that if we renormalize $\xi_{YX}$ by scalars $\lambda_{YX}\in \mathbf{k}^\times$
then $\mathcal J(\phi_Y)$ and $\mathcal J(\psi_X)$ change as follows:
$\mathcal J(\phi_Y)\mapsto \mathcal J'(\phi_Y)$ with
\begin{equation}\label{tran1}
\mathcal J'(\phi_Y)_{\Hom(W_3,W_1\otimes W_2)}=\mathcal J(\phi_Y)_{\Hom(W_3,W_1\otimes W_2)}\lambda_{YW_3}\lambda_{YW_1}^{-1}\lambda_{YW_2}^{-1}
\end{equation}
and
$\mathcal J(\psi_X)\mapsto \mathcal J'(\psi_X)$ with
\begin{equation}\label{tran2}
\mathcal J'(\psi_X)_{\Hom(V_3,V_1\otimes V_2)}=\mathcal J(\psi_X)_{\Hom(V_3,V_1\otimes V_2)}\lambda_{V_3X}\lambda_{V_1X}^{-1}\lambda_{V_2X}^{-1}.
\end{equation}

The following proposition gives the universal property of the twisted Deligne product $\E=\C\bullet \D$.

\begin{proposition}\label{tffromdp} Let $\B$ be a tensor category. Up to monoidal isomorphism, a tensor functor $F: \E\to \B$ is determined by the following data:

(1) A tensor functor $F_1: \C\to \B$;

(2) A tensor functor $F_2: \D\to \B$;

(3) Isomorphisms
$$
b_{YX}: F_2(Y)\otimes F_1(X)\cong F_1(X)\otimes F_2(Y),\ X\in {\rm Irr}(\C),Y\in {\rm Irr}(\D),
$$
 such that after extending them additively (hence functorially) to all objects and suppressing associators and the identity maps on the other tensor factors, we have
\[
\begin{aligned}
((T^1)^{-1}\otimes 1)\,b_{Y,W_1\otimes W_2}\,(1\otimes T^1)
  &=\overline J^\phi_Y b_{YW_2}b_{YW_1},\\
(1\otimes(T^2)^{-1})\,b_{V_1\otimes V_2,X}\,(T^2\otimes 1)
  &=b_{V_1X}b_{V_2X}\overline J^\psi_X .
\end{aligned}
\]
The first identity holds for simple $Y$ and all $W_1,W_2\in\C$, and the second for simple $X$ and all $V_1,V_2\in\D$. We write $T^1=T^1_{W_1,W_2}$ and $T^2=T^2_{V_1,V_2}$ for the tensor structures of $F_1$ and $F_2$. The correction factors are
\[
\begin{aligned}
\overline J^\phi_Y&=(T^1)^{-1}F_1(\mathcal J(\phi_Y)_{W_1,W_2})T^1,\\
\overline J^\psi_X&=(T^2)^{-1}F_2(\mathcal J(\psi_X)_{V_1,V_2})T^2 .
\end{aligned}
\] The crossings with the unit agree with the unit identifications.
\end{proposition}

\begin{proof} Restrict $F$ to $\C$ and $\D$ to obtain $F_1$ and $F_2$. For simple $X\in\C$ and $Y\in\D$, the tensor structure of $F$ identifies $F(X\otimes Y)$ with $F_1(X)\otimes F_2(Y)$ and $F(Y\otimes X)$ with $F_2(Y)\otimes F_1(X)$. Under these identifications, $F(\xi_{YX})$ gives the crossing $b_{YX}$. Applying $F$ to \eqref{xiiden} gives the two identities in (3).

Conversely, use the abelian equivalence $\C\boxtimes\D\simeq\E$ to define the additive functor $F$ by $F(WV)=F_1(W)\otimes F_2(V)$. Let
\[
 s:(W_1V_1)(W_2V_2)\longrightarrow (W_1W_2)(V_1V_2)
\]
be the isomorphism obtained by applying $\xi_{V_1W_2}$ to the two middle factors and rebracketing. Define the tensor structure $T$ of $F$ by
\[
 F(s)T_{W_1V_1,W_2V_2}
   =(T^1\otimes T^2)(1\otimes b_{V_1W_2}\otimes1),
\]
with the associators in $\B$ understood.

For three objects $X_rY_r$ ($r=1,2,3$), both sides of the tensor-functor axiom reorder the factors as $X_1X_2X_3Y_1Y_2Y_3$. On the left, $Y_1$ first crosses $X_2$, and then $Y_1Y_2$ crosses $X_3$. By the second identity in (3), this last crossing moves $Y_2$ past $X_3$ and then $Y_1$ past $X_3$. On the right, $Y_2$ first crosses $X_3$, and then $Y_1$ crosses $X_2X_3$. By the first identity in (3), $Y_1$ crosses $X_2$ and $X_3$ in that order.

Equation~\eqref{xiiden} gives the same expansions for the source crossings $\xi$. The tensor structure $\overline J^\psi_{X_3}$ therefore cancels with its inverse on the left, as does $\overline J^\phi_{Y_1}$ on the right. The two composites now differ only in the order of the crossings of $Y_1$ with $X_2$ and of $Y_2$ with $X_3$. These crossings commute by bifunctoriality, and the final crossing of $Y_1$ with $X_3$ is common to both composites. The tensor-functor axioms for $F_1,F_2$ identify the maps within each factor. Additivity gives the axiom for arbitrary triples.
\end{proof}

\begin{corollary} \label{dete}
The structure of the twisted Deligne product $\E=\C\bullet\D$ is completely determined by the functors $\phi_Y$ and $\psi_X$ (i.e., by $\mathcal J(\phi_Y)$ and $\mathcal J(\psi_X)$) for all simple $X\in \C$ and $Y\in \D$. Moreover, the tensor structures related by the transformation defined by \eqref{tran1},\eqref{tran2} give rise to the same twisted Deligne product (up to a tensor equivalence which is the
identity on $\C$ and $\D$).
\end{corollary}

\begin{proof} Apply Proposition~\ref{tffromdp} with $\B=\E'$, with the two given embeddings for $F_1,F_2$, and with the crossings of $\E'$ for $b$. Equality of the tensor structures makes (3) hold. The resulting functor sends each simple $XY$ to the corresponding simple $XY$, hence is an equivalence. Renormalizing the crossings gives the last assertion by \eqref{tran1} and \eqref{tran2}. Conversely, any marked equivalence transports one set of crossings to another. On a simple pair these differ by a scalar since $XY$ is simple. Thus \eqref{tran1} and \eqref{tran2} account for all changes induced by marked tensor equivalences.
\end{proof}

\subsection{Properties of \texorpdfstring{$\phi_Y$ and $\psi_X$}{phi(Y) and psi(X)}}

Let $Y_p$ be the simple objects of $\D$. Then for each $p,q,m$ we have a $\C$-bimodule functor
$$
F_{pq}^m: \C_{Y_p}\boxtimes_{\C}\C_{Y_q}\to \C_{Y_m}
$$
given by tensor product in $\E$ followed by projection to the $\C_{Y_m}$ summand. Under the natural identification of the source with $\C$ as a left $\C$-module,
it reduces to tensoring with the vector space $\Hom(Y_m,Y_p\otimes Y_q)$. Since
 $\C_Y\cong \C(\phi_Y)$, we have
 $$
 F_{pq}^m: \C(\phi_p)\boxtimes_\C \C(\phi_q)\to \C(\phi_m),
 $$
 where $\phi_p:=\phi_{Y_p}$. For automorphisms $F,G$ of $\C$ the balanced functor $(V,W)\mapsto V\otimes F(W)$ exhibits $\C(FG)$ as $\C(F)\boxtimes_{\C}\C(G)$. Indeed, a balanced functor on the two factors is recovered from its restriction to pairs $(T,\mathbf{1})$ by the balancing isomorphism $(V,W)\sim(V\otimes F(W),\mathbf{1})$. This gives the universal factorization, also on natural transformations.

\begin{proposition}\label{cocyc} (i) If $Y_m$ is contained in $Y_p\otimes Y_q$ then $\widehat{\phi}_m=\widehat{\phi}_p\widehat{\phi}_q$ modulo $\Im \theta$.
Thus, $\widehat\phi$ descends to a homomorphism
$$
\overline{\phi}: U_\D\to \Aut_1(\C)/H^2(U_\C,\mathbf{k}^\times).
$$

(ii) If $X_k$ is contained in $X_i\otimes X_j$ then $\widehat{\psi}_k=\widehat{\psi}_j\widehat{\psi}_i$ modulo $\Im \theta$. Thus, $\widehat\psi$ descends to an anti-homomorphism
$$
\overline{\psi}: U_\C\to \Aut_1(\D)/H^2(U_\D,\mathbf{k}^\times).
$$
\end{proposition}

\begin{proof} (i) Since
$\C(\phi_p)\boxtimes_\C \C(\phi_q)\cong \C(\phi_p\phi_q)$, the functor $F_{pq}^m$ may be viewed as a $\C$-bimodule functor $\C(\phi_p\phi_q)\to \C(\phi_m)$, i.e.,
it is an object of the twisted Drinfeld center $\Z_{\phi_m\phi_q^{-1}\phi_p^{-1}}(\C)$ which is a multiple of $\mathbf{1}$ as an object of $\C$. Since by Lemma \ref{inje}, $H^2(U_\C,\mathbf{k}^\times)\hookrightarrow \Aut_1(\C)$, (i) follows from Lemmas \ref{homom} and  \ref{twistedcenter}.

(ii) The proof is analogous, but with reversed order of factors. Write \({}_F\D\) for $\D$ with its left action twisted by $F$ and its usual right action. The relation $Y\otimes X\cong X\otimes\psi_X(Y)$ identifies $X\otimes\D$ with \({}_{\psi_X}\D\). Moving $Y$ first past $X_i$ and then past $X_j$ gives a bimodule functor \({}_{\psi_j\psi_i}\D\to{}_{\psi_k}\D\) whenever $X_k\subset X_i\otimes X_j$. The equivalence \({}_F\D\simeq\D(F^{-1})\) is given by $F^{-1}$. Applying Lemma~\ref{twistedcenter} as in (i) gives the stated antihomomorphism.
\end{proof}

\begin{corollary}\label{multfree} If $\D$ is multiplicity free, then $\widehat\phi$ descends to a homomorphism $U_\D\to\Aut_1(\C)$. The same holds if $H^2(U_\C,\k^\times)$ is finite and its order is coprime to every nonzero fusion multiplicity in $\D$. When $U_\C$ is finite, it suffices that the nonzero multiplicities be coprime to $|U_\C|$.
\end{corollary}

\begin{proof} For $Y_m\subset Y_p\otimes Y_q$, put $V=\Hom(Y_m,Y_p\otimes Y_q)$ and $d=\dim V$. By the proof of Proposition~\ref{cocyc}(i) and Lemma~\ref{twistedcenter}, $V$ carries a projective representation $\rho$ of $U_\C$ with multiplier $\beta$ such that $\theta([\beta])=\widehat\phi_m(\widehat\phi_p\widehat\phi_q)^{-1}$. Taking determinants in $\rho(gh)=\rho(g)\rho(h)\beta(g,h)$ gives
\[
 \beta(g,h)^d=\frac{\det\rho(gh)}{\det\rho(g)\det\rho(h)}.
\]
Thus $\beta^d$ is a coboundary and $[\beta]^d=1$. If $U_\C$ is finite, restriction and corestriction give $[\beta]^{|U_\C|}=1$ as well. Under any of the stated hypotheses the class is therefore trivial. Lemma~\ref{homom} now gives the homomorphism.
\end{proof}

Fix $X_k\subset X_i\otimes X_j$, $Y_m\subset Y_p\otimes Y_q$. Let $V:=\Hom(X_k,X_i\otimes X_j)$ and $W:=\Hom(Y_m,Y_p\otimes Y_q)$. In the next formula a product of tensor functors means composition. For identity underlying functors our tensor structure convention gives $J(FG)=J(G)J(F)$ on a multiplicity space.

\begin{proposition}\label{phipsi}
We have the following equality of scalars on $V$ and $W$, respectively:
\begin{equation}\label{equall}
\phi_m^{-1}\phi_p\phi_q|_{\Hom(X_k,X_i\otimes X_j)}=\psi_k^{-1}\psi_j\psi_i|_{\Hom(Y_m,Y_p\otimes Y_q)}.
\end{equation}
\end{proposition}

\begin{proof}
Let $A_r=J(\phi_r)|_V$ and $B_s=J(\psi_s)|_W$. Use the natural crossing $\xi_{Y_pY_q,X_iX_j}$ to identify source and target on the $(k,m)$ block. In these coordinates it acts as the identity on $V\otimes W$. Let $L$ be the composite of the four individual crossings sorting $Y_pY_qX_iX_j$ into $X_iX_jY_pY_q$. The two orders of sorting agree, since the interchanged steps are disjoint.

Expand the crossing first in the $Y$ variable using \eqref{xiiden}. This contributes $B_k$ at the source. Expanding the crossing of $Y_q$ with $X_iX_j$ introduces $A_q$; slide it through the entire natural crossing of $Y_p$ with $X_iX_j$, then expand this latter crossing to introduce $A_p$. Thus
\[
  1=(A_q A_p\otimes 1) L(1\otimes B_k).
\]
These formulas may first be read on the full direct sum over $k,m$: all factors other than $L$ preserve that decomposition, so $L$ does as well. Expanding first in the $X$ variable instead contributes $A_m$ at the target. Slide the correction $B_j$ through the entire crossing with $X_i$ before expanding it. This gives
\[
  1=(A_m\otimes 1)L(1\otimes B_i B_j).
\]
Solving both equations for $L$ and inverting gives 
$$
A_qA_p\otimes B_k=A_m\otimes B_iB_j.
$$ 
Hence $A_qA_p A_m^{-1}\otimes 1=1\otimes B_iB_jB_k^{-1}$. If $V,W\ne0$ and $P\otimes1=1\otimes Q$, applying a linear functional to the $W$ factor at a vector on which it takes value one shows that $P$ is scalar; substitution gives the same scalar for $Q$. The two scalar operators are exactly the tensor structures in \eqref{equall}.
\end{proof}

\subsection{The action of \texorpdfstring{$\Gamma_3(U_\C,U_\D,\k^\times)$}{the mixed cohomology group} on twisted Deligne products}

\begin{lemma}\label{act0} Let ${\mathcal T}(\C,\D)$ be the set of marked equivalence classes of twisted Deligne products of $\C$, $\D$.
Then the group $\Gamma_3(U_\C,U_\D,\k^\times)$ acts naturally on ${\mathcal T}(\C,\D)$. Moreover, this action is free.
\end{lemma}

\begin{proof} Recall that $\Gamma_3(U_\C,U_\D,\k^\times)$
is the kernel of the natural map $H^3(U_\C\times U_\D,\k^\times)\to H^3(U_\C,\k^\times)\oplus H^3(U_\D,\k^\times)$. The degree of a simple $XY$ is $(\deg X,\deg Y)$, giving a $U_\C\times U_\D$-grading because the two factors commute in the Grothendieck ring. The action of $\omega\in \Gamma_3(U_\C,U_\D,\k^\times)$ on ${\mathcal T}(\C,\D)$ is by multiplying the associativity morphism on homogeneous objects of degrees $(h_1,g_1),(h_2,g_2),(h_3,g_3)$ by the scalar $\omega((h_1,g_1),(h_2,g_2),(h_3,g_3))$. The restrictions of $\omega$ to $U_\C$ and to $U_\D$ are trivial in cohomology, so after replacing $\omega$ by a cohomologous cocycle if necessary the tensor structures on the embedded subcategories $\C$ and $\D$ are unchanged, so the action is well defined. If two such representatives differ by $du$, the restrictions of $u$ are cocycles. Dividing $u$ by their pullbacks from the factors makes both restrictions trivial without changing $du$. The resulting equivalence is marked. We prove freeness after Proposition~\ref{act}, using its explicit formula.
\end{proof}

Recall the maps of Example~\ref{ex1}(2):
\[
\begin{aligned}
\tau_1:\Gamma_3(U_\C,U_\D,\k^\times)&\longrightarrow H^1(U_\D,H^2(U_\C,\k^\times)),\\
\tau_2:\Gamma_3(U_\C,U_\D,\k^\times)&\longrightarrow H^1(U_\C,H^2(U_\D,\k^\times)).
\end{aligned}
\]
For a normalized representative $\omega$, regard $h,h_1,h_2\in U_\C$, $g,g_1,g_2\in U_\D$ as elements of the product via the two inclusions. The signed shuffles give the representatives
\[
\begin{aligned}
\tau_1(\omega)(g)(h_1,h_2)
 &=\frac{\omega(g,h_1,h_2)\,\omega(h_1,h_2,g)}{\omega(h_1,g,h_2)},\\
\tau_2(\omega)(h)(g_1,g_2)
 &=\frac{\omega(g_1,h,g_2)}{\omega(g_1,g_2,h)\,\omega(h,g_1,g_2)}.
\end{aligned}
\]

\begin{proposition}\label{transfo}
Let $\omega\in \Gamma_3(U_\C,U_\D,\k^\times)$. Then
the action of $\omega$ from Lemma \ref{act0} transforms $\widehat{\phi}$ and $\widehat{\psi}$ as follows:
$$
\widehat{\phi}\mapsto \widehat{\phi}\cdot \tau_1(\omega);\ \widehat{\psi}\mapsto \widehat{\psi}\cdot \tau_2(\omega).
$$
\end{proposition}

\begin{proof}
The functors $\phi_Y$ and $\psi_X$ are unchanged as additive functors; only their tensor structures change.  For $Y\in\D_g$, compare the first diagram in \eqref{xiiden} before and after multiplying the associator by $\omega$.  The three associators in that diagram contribute respectively
\[
\begin{gathered}
\omega((1,g),(h_1,1),(h_2,1))^{-1},\quad
\omega((h_1,1),(1,g),(h_2,1)),\\
\omega((h_1,1),(h_2,1),(1,g))^{-1}.
\end{gathered}
\]
Therefore $\mathcal J(\phi_Y)$ is multiplied by the inverse of this product, i.e. by $\tau_1(\omega)(g)(h_1,h_2)$.

Similarly, for $X\in\C_h$, the second diagram in \eqref{xiiden} shows that $\mathcal J(\psi_X)$ is multiplied by \scriptsize
\[
\omega((1,g_1),(1,g_2),(h,1))^{-1}
\omega((1,g_1),(h,1),(1,g_2))
\omega((h,1),(1,g_1),(1,g_2))^{-1},
\] \normalsize
which is $\tau_2(\omega)(h)(g_1,g_2)$ with the convention above.  Passing to classes in $\Aut_1$ gives the stated formulas.
\end{proof}

Note that this induced action on $\widehat\phi,\widehat\psi$ is, in general, not faithful, since the kernel $K:=K(U_\C,U_\D)$ of $\tau_1\oplus \tau_2$ (described in Example \ref{ex1}(2)) acts trivially on these classes.  Thus $\widehat\phi$ and $\widehat\psi$ do not detect the whole $\Gamma_3$-action.  To see the action of $K$, one must keep the actual tensor structures $\mathcal J(\phi_Y)$ and $\mathcal J(\psi_X)$, modulo the renormalizations \eqref{tran1}, \eqref{tran2}.  By Corollary~\ref{dete}, these data determine the twisted product up to equivalence, so it remains to compute how $K$ changes these representatives.

By Example \ref{ex1}(2), the group $K$ may be interpreted as the kernel of the natural map
$\eta: H^2(U_\D,H^1(U_\C,\mathbf{k}^\times))\to H^1(U_\C,H^2(U_\D,\mathbf{k}^\times))$ given on cocycles by
$\eta(\gamma)(h)(g_1,g_2)=\gamma(g_1,g_2)(h)$, $g_1,g_2\in U_\D, h\in U_\C$.
Thus $K$ consists of images $\gamma_*$ of 2-cocycles $\gamma$ on $U_\D$ with coefficients
in $\Hom(U_\C,\mathbf{k}^\times)$ such that $\gamma(-,-)(h)$ is a coboundary
for each $h\in U_\C$, but $\gamma$ itself is not necessarily a coboundary.
Let $\lambda(h,g)$ be a function on $U_\C\times U_\D$ such that
for each $h$, the cocycle $\gamma(-,-)(h)$ is the coboundary of $\lambda(h,-)$. Choose $\gamma$ normalized and take $\lambda(1,g)=\lambda(h,1)=1$, so that renormalizing a crossing preserves the unit constraints.
Two such choices of $\lambda$ differ by a function $\mu(h,g)$ such that $\mu(h,-)$ is a character of $U_\D$ for each $h$ and $\mu(1,-)=1$.

\begin{proposition}\label{act} The action of $\gamma$ on $\phi$ and $\psi$ in Lemma \ref{act0} (up to equivalence) is as follows:
$\psi$ remains the same and, for $Y\in \D_g$, the tensor structure $\mathcal J(\phi_Y)$ on the block $\Hom(X_3,X_1\otimes X_2)$ is multiplied by the scalar
$\lambda(h_1,g)\lambda(h_2,g)/\lambda(h_1h_2,g),\ h_a=\deg X_a.$
\end{proposition}

\begin{proof}
In the normalized bar model we may represent the class of $\gamma$ by the $3$-cocycle
\[
        \omega_\gamma((h_1,g_1),(h_2,g_2),(h_3,g_3))=\gamma(g_1,g_2)(h_3).
\]
With this representative, $\tau_1(\omega_\gamma)=1$ and
$\tau_2(\omega_\gamma)(h)(g_1,g_2)=\gamma(g_1,g_2)(h)^{-1}$.
Choose $\lambda$ so that
\[
        \gamma(g_1,g_2)(h)=\lambda(h,g_1)\lambda(h,g_2)\lambda(h,g_1g_2)^{-1}.
\]
Then the change in $\mathcal J(\psi_X)$ is removed by the renormalization
$$
        \xi_{YX}\mapsto \lambda(\deg X,\deg Y)^{-1}\xi_{YX}.
$$
Using \eqref{tran2}, this renormalization multiplies $\mathcal J(\psi_X)$ by $d\lambda(h,-)$, so $\psi$ returns to its original representative.  By \eqref{tran1}, the same renormalization changes $\mathcal J(\phi_Y)$ by the factor
$
        \lambda(h_1,g)\lambda(h_2,g)/\lambda(h_1h_2,g),
$
as claimed.
\end{proof}

Note that if we change $\lambda$ to another representative $\lambda'$, then $\lambda/\lambda'=\mu$ is a character with
respect to $g$. Hence by Corollary \ref{dete}, changing $\phi$ by $\mu(h_1,g)\mu(h_2,g)/\mu(h_1h_2,g)$
is a trivial transformation.
Also, if $\gamma$ is a coboundary
then $\lambda$ can be chosen a character with respect to $h$, so the action we defined is trivial.
Thus the action of Proposition \ref{act} is well defined.

We can now finish the proof of freeness in Lemma~\ref{act0}. If a marked equivalence identifies a product with its twist by $\omega$, Proposition~\ref{transfo} and injectivity of $\theta$ on both factors give $\tau_1(\omega)=\tau_2(\omega)=1$. Represent its class in $K$ by $\gamma=d_{U_\D}\lambda$ as above. By Corollary~\ref{dete} the marked equivalence acts on crossings by scalars $\nu_{Y,X}$. Choose the direction that undoes the twist. Since $J(\psi_X)$ has been restored, \eqref{tran2} and Lemma~\ref{homom} give $\nu_{Y,X}=\nu_X(\deg Y)$ for a character $\nu_X$ of $U_\D$. By \eqref{tran1} and Proposition~\ref{act}, the function $X\mapsto\lambda(\deg X,g)/\nu_X(g)$ is multiplicative on fusion constituents for each $g$. It therefore defines $f_g\in\Hom(U_\C,\k^\times)$. Since each $\nu_X$ is a character, $(d_{U_\D}f)(g_1,g_2)(h)=\gamma(g_1,g_2)(h)$. Thus $[\gamma]=1$ in the character-valued cohomology group itself, proving freeness.

\begin{example} Suppose that $U_{\mathcal C}$ and $U_{\mathcal D}$ are finite and $\k=\mathbb C$.
Put $A=(U_{\mathcal C})_{\rm ab}$ and $B=(U_{\mathcal D})_{\rm ab}$.
The K\"unneth theorem, together with the universal coefficient theorem, identifies the kernel
$K=\ker(\tau_1\oplus \tau_2)$
with
$\operatorname{Hom}(\operatorname{Tor}(A,B),\mathbb C^\times)$.

Equivalently, since the exponential map gives a canonical identification
$\mathbb Q/\mathbb Z \xrightarrow{\sim} \mu_\infty\subset \mathbb C^\times,
\ t\mapsto e^{2\pi i t}$,
and all characters of finite abelian groups have values in $\mu_\infty$, we may write
\[
K\cong \operatorname{Tor}(A,B)^*
\cong A^*\otimes B^*
\cong H^1(U_{\mathcal C},\mathbb C^\times)\otimes H^1(U_{\mathcal D},\mathbb C^\times),
\]
where for a finite abelian group $L$, $L^*:=\operatorname{Hom}(L,\mathbb Q/\mathbb Z)$, which is
identified with $\Hom(L,\mathbb C^\times)$ by the exponential map. The last identification uses the displayed exponential convention. Indeed, the evaluation map on $H^2(U_\D,A^*)$ has kernel $\operatorname{Ext}^1_{\mathbb Z}(B,A^*)$ by the universal coefficient theorem. Since $B$ is finite, a finite free presentation identifies this kernel with $A^*\otimes\operatorname{Ext}^1_{\mathbb Z}(B,\mathbb Z)=A^*\otimes B^*$.

In terms of this realization, the map $\gamma\mapsto\lambda$ can be described (on decomposable tensors) as follows.
Let \(c:U_\C\to\mathbb C^\times\) and \(d:U_\D\to\mathbb C^\times\) be characters and $\gamma_*=c\otimes d$.
  Choose functions \(L_c,L_d\), with \(L_c(1)=L_d(1)=0\), such that \(c(h)=\exp L_c(h)\) and \(d(g)=\exp L_d(g)\), and set
\[
        \lambda(h,g)=\exp\!\left(\frac{L_c(h)L_d(g)}{2\pi i}\right).
\]
Then
\[
        \gamma(g_1,g_2)(h)=
        \lambda(h,g_1)\lambda(h,g_2)\lambda(h,g_1g_2)^{-1}
\]
is a representative of the corresponding element of \(K\).

The dependence on the choices of logarithms is as follows. If
$L_c'(h)=L_c(h)+2\pi iM(h)$, then
\[
        \frac{\lambda'(h,g)}{\lambda(h,g)}=d(g)^{M(h)},
\]
which is a character in the $g$-variable for every fixed $h$ and is
therefore exactly the ambiguity described above. If instead
$$
L_d'(g)=L_d(g)+2\pi iN(g),
$$ 
then
\[
        \nu(h,g):=\frac{\lambda'(h,g)}{\lambda(h,g)}=c(h)^{N(g)}.
\]
For every fixed $g$, the function $\nu(-,g)$ is a character of
$U_\C$, so $d_{U_\C}\nu=1$. Hence the factor
$d_{U_\C}\lambda$ governing the change of $\phi$ is unchanged, while
\[
        \gamma'=d_{U_\D}\lambda'
        =\gamma\,d_{U_\D}\nu
\]
is cohomologous to $\gamma$ as a $U_\D$-cocycle with values in
$H^1(U_\C,\mathbb C^\times)$. Thus the resulting element of $K$ and
the induced action are independent of the choices of $L_c$ and $L_d$.
\end{example}

\begin{proposition}\label{equii} Let $\E$, $\E'$ be two twisted products with the same $\widehat{\phi}$ and $\widehat{\psi}$.

(i) If $\C$ or $\D$ is a pointed category then $\E$ and $\E'$ differ by twisting the associativity by an element of $K(U_\C,U_\D)$.

(ii)  If $U_\C$ or $U_\D$ is trivial then $\E\cong \E'$.
\end{proposition}

\begin{proof} Let $\phi,\psi$ be the maps corresponding to $\E$ and $\phi',\psi'$ to $\E'$.
Renormalizing $\xi_{YX}$, we may assume that $\psi=\psi'$.
Then
$$
\phi_Y^{-1}\phi_Y'|_{\Hom(W_3,W_1\otimes W_2)}=\lambda_{YW_1}\lambda_{YW_2}\lambda_{YW_3}^{-1}.
$$
By Proposition \ref{phipsi}, this is multiplicative in $Y$: indeed, applying Proposition~\ref{phipsi} to $\E$ and to $\E'$ and dividing the two equalities cancels the common $\psi$-defect.  Hence by Lemma \ref{homom}
$$
\lambda_{YW_1}\lambda_{YW_2}\lambda_{YW_3}^{-1}=\zeta(W_1,W_2,W_3)(\deg Y),
$$
where $\zeta(W_1,W_2,W_3)\in \Hom(U_\D,\mathbf{k}^\times)$.

(i) If $\C$ is pointed, then
$\zeta$ depends only on $W_1$ and $W_2$ (as $W_3=W_1\otimes W_2$) and is a 2-cocycle on $U_\C$ with coefficients in $\Hom(U_\D,\mathbf{k}^\times)$,
such that for every $g\in U_\D$, the 2-cocycle $\zeta(-,-)(g)$ is a coboundary.
 More explicitly, write the simple objects of $\C$ as $X_h$, $h\in U_\C$, and set (choosing one simple $Y_g$ of each degree $g$, with $Y_1=\mathbf{1}$)
\[
        \lambda(h,g):=\lambda_{Y_gX_h},\qquad g\in U_\D.
\]
By the definition of $\zeta$,
\[
        \zeta(h_1,h_2)(g)
        =\lambda(h_1,g)\lambda(h_2,g)\lambda(h_1h_2,g)^{-1}.
\]
Since $\zeta(h_1,h_2)$ is a character of $U_\D$, the cochain
\[
        \gamma(g_1,g_2)(h):=
        \lambda(h,g_1)\lambda(h,g_2)\lambda(h,g_1g_2)^{-1}
\]
takes values in $H^1(U_\C,\mathbf{k}^\times)$: indeed $d_{U_\C}\gamma=d_{U_\D}\zeta=1$.  Also $d_{U_\D}\gamma=1$, so $\gamma$ is a $2$-cocycle on $U_\D$ with coefficients in $H^1(U_\C,\mathbf{k}^\times)$.  For each fixed $h$, the cocycle $\gamma(-,-)(h)$ is the coboundary $d_{U_\D}\lambda(h,-)$, so the image of $[\gamma]$ under the transposition map $\eta$ of Example~\ref{ex1}(2) is $1$.  Hence $[\gamma]\in K(U_\C,U_\D)$.  Proposition~\ref{act}, applied to this representative and to the cochain $\lambda$, leaves $\psi$ unchanged and, for every simple $Y$ of degree $g$, multiplies the tensor structure of $\phi_Y$ on $X_{h_1}\otimes X_{h_2}$ by
\[
        \lambda(h_1,g)\lambda(h_2,g)\lambda(h_1h_2,g)^{-1},
\]
which is exactly the scalar relating the tensor structures of $\E$ and $\E'$.  Therefore the two twisted products differ by the action of the class $[\gamma]\in K(U_\C,U_\D)$.

 If $\D$ is pointed, interchange the two factors and use the inverse crossings. Interchanging the two edge maps preserves their common kernel $K$.

(ii) If $U_\D$ is trivial then $\zeta=1$ and the statement follows, and if $U_\C$ is trivial then the proof is symmetrical.
\end{proof}

\subsection{Examples}

\begin{example} If $\C=\Vec_G,\D=\Vec_H$, then $U_\C=G$, $U_\D=H$, $\Aut_1(\C)=H^2(G,\mathbf{k}^\times)$, $\Aut_1(\D)=H^2(H,\mathbf{k}^\times)$.
Since both categories are multiplicity free, by Corollary \ref{multfree} $\widehat{\phi}$ and $\widehat{\psi}$ are homomorphisms
$U_\D\to H^2(U_\C,\mathbf{k}^\times)$ and $U_\C\to H^2(U_\D,\mathbf{k}^\times)$, respectively.

Also, because of the action of $\Gamma_3(G,H,\k^\times)$,
any pair of homomorphisms can arise, as $\tau_1\oplus \tau_2$ is surjective.
Indeed, $\k^\times$ is divisible because $\k$ is algebraically closed. Thus for $L=G,H$ the universal coefficient theorem gives
\[
 H^2(L,\k^\times)\cong\Hom(H_2(L,\mathbb Z),\k^\times).
\]
Consider the transposition map from Example~\ref{ex1}(2),
\[
        \eta:H^2(G,\Hom(H,\k^\times))\to \Hom(H,H^2(G,\k^\times)).
\]
The universal coefficient exact sequence gives a surjection
\[
        H^2(G,\Hom(H,\k^\times))\twoheadrightarrow
        \Hom(H_2(G,\mathbb Z),\Hom(H,\k^\times)).
\]
After currying, the target is canonically
\[
        \textstyle\Hom(H,\Hom(H_2(G,\mathbb Z),\k^\times))=
        \Hom(H,H^2(G,\k^\times)),
\]
and this is exactly the map induced by $\eta$.  Thus $\eta$ is surjective.  Now choose a class in $\Gamma_3(G,H,\k^\times)$ with the prescribed value of one edge map; the difference between its other edge-map value and the desired one can be corrected by adding an element in the kernel of the first edge map, and the preceding surjectivity of $\eta$ realizes this correction.  Hence $\tau_1\oplus\tau_2$ is surjective.

 Finally, once these homomorphisms
have been fixed, by Proposition \ref{equii} the possible twisted products are parametrized by a torsor over the group $K$. This torsor is in fact nonempty, since in this case twisted products $\E=\C\bullet \D$ are parametrized by $\Gamma_3(G,H,\k^\times)$ (Example \ref{groupprod}).

More generally, if $\C=\Vec_G^{\omega_G},\D=\Vec_H^{\omega_H}$, the story is the same, with associator of $\E$ modified by $\omega_G$ and $\omega_H$. This recovers the twisted Deligne products in Example~\ref{groupprod}.
\end{example}

\begin{example} \label{smashpro}
Let \(\D=\Vec_G\), and denote the simple invertible object corresponding to \(g\in G\) again by \(g\).  If \(\E=\C\bullet \Vec_G\) is a twisted Deligne product, then conjugation by \(g\),
$T_g(X)=g\otimes X\otimes g^{-1}$ for $X\in\C$, defines a categorical action $T$ of $G$ on $\C$ preserving isomorphism classes of simple objects.

Conversely, any categorical action $T$ of $G$ on $\C$ fixing the isomorphism classes of simple objects defines a semidirect product \(\C\rtimes \Vec_G\) by the construction of \cite[Definition~4.15.5]{EGNO}. For arbitrary $G$ one uses finite-support $G$-graded objects and the same product $(X,g)(Y,h)=(X\otimes T_g(Y),gh)$.  Its underlying abelian category is \(\C\boxtimes \Vec_G\),
 and its tensor product is characterized by the crossing relation
$g\otimes X=T_g(X)\otimes g$.
This category contains \(\Vec_G\) and \(\C\) as tensor subcategories, and it is a twisted Deligne product of them.  The functor $(X,g)\mapsto X\otimes g$ is a marked equivalence from the semidirect product for the conjugation action to $\E$. Its tensor structure uses $g\otimes Y\cong T_g(Y)\otimes g$ and the multiplication in the embedded $\Vec_G$. Conversely, conjugation by $(\mathbf{1},g)$ in the semidirect product recovers $T_g$ with its given composition constraints.

The composition isomorphisms $\mu_{g,h}:T_gT_h\to T_{gh}$ give $[T_g][T_h]=[T_{gh}]$ in $\Aut_1(\C)$. Thus the first invariant is the homomorphism $\widehat\phi:g\mapsto[T_g]$. Identify $T_gX$ with $X$ by $u_{g,X}$ for simple $X$. Take $u_{1,X}$ to be the unit identification and define $s_X$ by
\[
 u_{g,X}T_g(u_{h,X})=s_X(g,h)\,u_{gh,X}\mu_{g,h,X}.
\]
The action axiom gives $d s_X=1$. Rescaling the $u_{g,X}$ changes $s_X$ by a coboundary. Its class is the obstruction to equipping $X$ with an equivariant structure. In the crossed product the associator on $g,h,X$ uses $\mu_{g,h,X}^{-1}$, so the sequential crossings have scalar $s_X(g,h)$ relative to the single crossing. The second identity in \eqref{xiiden} gives $\widehat\psi_X=[s_X]^{-1}$.

Fix a homomorphism $\widehat\phi$ and representatives $T_g$. Monoidal automorphisms of the identity form $M=\Hom(U_\C,\k^\times)$ by Lemma~\ref{homom}; the induced $G$-action on $M$ is trivial since the $T_g$ fix the simple classes. Choose monoidal isomorphisms $T_gT_h\to T_{gh}$. The ratio of their two triple composites is an $M$-valued $3$-cocycle (compare the five composites on four inputs); changing the isomorphisms multiplies it by a coboundary. Its class in $H^3(G,M)$ is the obstruction to a categorical lift. If it vanishes, fix one lift $T^0$. All lifts are then parametrized, up to isomorphism, by $H^2(G,M)$: multiply its composition constraints by a character-valued $2$-cocycle, modulo rescaling by a character-valued $1$-cochain.

Equivalently, act on this fixed reference product by the kernel of $\tau_1$, identified with $H^2(G,M)$ by Example~\ref{ex1}. For any normalized $\gamma\in Z^2(G,M)$ the formula $\omega_\gamma((a,g),(b,h),(c,l))=\gamma(g,h)(c)$ represents the corresponding class in $\ker\tau_1$, as in Example~\ref{ex1}. Under this twist the crossed-product associator is multiplied by $\gamma(g,h)(\deg X)$ on $g,h,X$; this replaces $\mu^0_{g,h,X}$ by $\gamma(g,h)(\deg X)^{-1}\mu^0_{g,h,X}$. With the same identifications $u_{g,X}$, the defining equation for $s_X$ therefore multiplies $s_X(g,h)$ by $\gamma(g,h)(\deg X)$. Since $\widehat\psi_X=[s_X]^{-1}$, we get
\[
 \widehat\psi_X(\E^{\omega_\gamma})\,\widehat\psi_X(\E)^{-1}
   =[\gamma(-,-)(\deg X)]^{-1},
\]
where $\E$ is the fixed reference product. The ratio factors through $U_\C$. Every homomorphism $U_\C\to H^2(G,\k^\times)$ occurs as such a ratio: the evaluation map $H^2(G,M)\to\Hom(U_\C,H^2(G,\k^\times))$ is surjective by the universal coefficient argument of the preceding example. Its kernel is $K(U_\C,G)$. For instance, with $\C=\Vec_{\mathbb Z/2}$, $G=(\mathbb Z/2)^2$ and $\k=\mathbb C$, twisting the ordinary product by $\gamma(g,h)(a)=(-1)^{a g_1h_2}$ changes $\widehat\psi_{X_1}$ by the nontrivial class of $(-1)^{g_1h_2}$. This cocycle is nonsymmetric on the abelian group $G$.

More generally, we can take $\D=\Vec_G^\alpha$ where $\alpha: G^3\to \k^\times$ is a 3-cocycle. Then the story is the same, with associator of $\E$ modified by $\alpha$.
\end{example}

\begin{remark} Example \ref{smashpro} shows that the action of Proposition \ref{transfo} need not be transitive. Take $\k=\mathbb C$. For a finite group $H$ let ${\rm Aut}_c(H)$ be the group
of automorphisms of $H$ which act trivially on the set of conjugacy classes of $H$, and
let ${\rm Inn}(H)$ be its normal subgroup of inner automorphisms. As explained in \cite{GK},
(proof of Proposition 1.4), there exist $p$-groups $H$ for which  ${\rm Aut}_c(H)/{\rm Inn}(H)$
is non-abelian. Let $G:={\rm Aut}_c(H)$, and consider the twisted Deligne product
$\Rep(H)\rtimes \Vec_G$. Then we have a natural homomorphism $\overline \phi: G\to \Aut_1(\Rep(H))/{\rm Im}\theta$, which is nontrivial since $G/{\rm Inn}(H)\subset \Aut_1(\Rep(H))$ is non-abelian. However, $\overline \phi$ does not change under the action of $\Gamma_3$.
Thus the twisted Deligne product $\Rep(H)\rtimes \Vec_G$ is not equivalent under this action
to the usual Deligne product $\Rep(H)\boxtimes \Vec_G$.
\end{remark}

\subsection{Classification theorems}

We apply Proposition~\ref{equii} to compare a twisted Deligne product with the ordinary one.

\begin{theorem}\label{part1} Suppose that $U_\C=U_\D=1$. Then any twisted Deligne product $\E=\C\bullet \D$ is equivalent to the ordinary Deligne product $\C\boxtimes \D$.
\end{theorem}

\begin{proof} Since $U_\C=U_\D=1$, the maps $\overline{\phi}$ and $\overline{\psi}$ are trivial.
Also the groups $H^2(U_\C,\mathbf{k}^\times)$ and $H^2(U_\D,\mathbf{k}^\times)$ are trivial, so the maps $\widehat{\phi}$ and $\widehat{\psi}$ are trivial.
Proposition~\ref{equii}(ii) applied to $\E$ and $\C\boxtimes\D$ gives the marked equivalence.
\end{proof}

If one factor has trivial universal grading group and trivial $\Aut_1$, the other factor is unrestricted.

\begin{theorem}\label{part2} Assume that $U_\C=1$ and $\Aut_1(\C)=1$, or $U_\D=1$ and $\Aut_1(\D)=1$. Then any twisted Deligne product $\E=\C\bullet \D$ is equivalent to the ordinary Deligne product $\C\boxtimes \D$.
\end{theorem}

\begin{proof} Let $U_\C=1$ and $\Aut_1(\C)=1$ (the other case is symmetrical).
Then we have $\widehat\phi=1$. For each simple $Y$, choose a monoidal identification of $\phi_Y$ with the identity. Its components on simple $X$ are scalars. Renormalizing the crossings for this $Y$ by \eqref{tran1}, we obtain $J(\phi_Y)=1$ simultaneously for all $Y$. Thus by Proposition \ref{phipsi}, $\widehat \psi_k^{-1}\widehat \psi_j\widehat \psi_i=1$ for $X_k\subset X_i\otimes X_j$.
Hence by Lemma \ref{homom}, $\widehat \psi$ is an antihomomorphism $U_\C\to \Aut_1(\D)$, so $\widehat \psi=1$ since $U_\C=1$. Thus, again using that $U_\C=1$, the result follows from Proposition \ref{equii}(ii).
\end{proof}

\section{Appendix: based-ring cocycles and categorical cocycles}

In this appendix we study based-ring cocycles and categorical cocycles
 in tensor categories, in particular answering the rigid case of Question 4.3 in \cite{JOY}. Although the results below are not used in the main body of the paper, they are motivated
by some of its arguments, e.g. Lemma \ref{twistedcenter} and its proof. 

\subsection{Based-ring cocycles}
Let $R$ be a based ring with basis $I$ and structure constants $N_{xz}^y$, and let
$A$ be an abelian group written multiplicatively. For $n\ge 1$, we call a function
$f:I^n\to A$ a \emph{based-ring $n$-cocycle} if for every $x_1,\ldots,x_{n+1}\in I$ and every choice of $y_i\in I$ with $N^{y_i}_{x_ix_{i+1}}>0$, $1\le i\le n$, one has \scriptsize
\begin{equation}\label{frcocycle}
 f(x_2,\ldots,x_{n+1})
 \biggl(\prod_{i=1}^{n}
 f(x_1,\ldots,x_{i-1},y_i,x_{i+2},\ldots,x_{n+1})^{(-1)^{i}}\biggr)
 f(x_1,\ldots,x_{n})^{(-1)^{n+1}}=1 .
\end{equation} \normalsize

\begin{proposition}\label{anydegreecocycle} Let $f:I^n\to A$ be a based-ring $n$-cocycle. Then there is a unique group $n$-cocycle
$\overline f:U(R)^n\to A$
such that
$$
 f(x_1,\ldots,x_n)=\overline f(\deg x_1,\ldots,\deg x_n)
$$
for all $x_1,\ldots,x_n\in I$.
\end{proposition}

\begin{proof} It suffices to show that $f$ is constant on universal grading classes in each variable.  For
$1\le p\le n$ and $m\ge 1$, let $P(p,m)$ be the following assertion: if
$x,x'\le u_1\cdots u_m$, then, for all $1\le q\le p$ and fixed entries
$a_1,\ldots,a_{q-1},a_{q+1},\ldots,a_n$, one has
\begin{equation*}
        f(a_1,\ldots,a_{q-1},x,a_{q+1},\ldots,a_n)
        =
        f(a_1,\ldots,a_{q-1},x',a_{q+1},\ldots,a_n).
\end{equation*}
We shall prove $P(p,m)$ for all $p,m$ by induction on $p$ (starting with the vacuous case $p=0$), and, for each fixed $p$, by induction on $m$.

The case $m=1$ is trivial: $x=x'=u_1$.  So consider $m>1$.
We assume that $P(q,r)$ has already been proved for all
$q<p$ and all $r$, and that $P(p,r)$ has already been proved for all $r<m$.

Let
$x,x'\le u_1\cdots u_m$.  Choose elements
$t,t'\le u_1\cdots u_{m-1}$ such that
$
        x\le tu_m,
        \
        x'\le t'u_m .
$
Apply \eqref{frcocycle} to the $(n+1)$-tuple
$$
        (a_1,\ldots,a_{p-1},t,u_m,a_{p+1},\ldots,a_n)
$$
with $y_{p}=x$.  Apply it again to
$$
        (a_1,\ldots,a_{p-1},t',u_m,a_{p+1},\ldots,a_n)
$$
with $y_{p}=x'$.  For $i\ne p$, choose the adjacent summands in the two equations as follows.  If the adjacent pair does not involve $t$ or $t'$, choose the same summand in the two equations.  If $i=p-1$ (when $p>1$), choose arbitrary summands $
        y_{p-1}\le a_{p-1}t,
        \
        y'_{p-1}\le a_{p-1}t'.
$

We compare all factors except the one with $i=p$. For $p=1$, the initial factor $f(x_2,\ldots,x_{n+1})$ is the same in both equations. For $p>1$, its only change is from $t$ to $t'$ in position $p-1$, so $P(p-1,m-1)$ gives equality. The same argument applies to each middle factor with $i<p-1$, since replacing the adjacent pair by its chosen summand again puts $t$ or $t'$ in position $p-1$.

For $i=p-1$ with $p>1$, the two summands $y_{p-1}$ and $y'_{p-1}$ both occur in $a_{p-1}u_1\cdots u_{m-1}$. They occupy position $p-1$, so $P(p-1,m)$ gives equality of these factors.

For $i>p$, the only change in a middle factor is from $t$ to $t'$ in position $p$. The same holds for the final factor $f(x_1,\ldots,x_n)$. Since both $t$ and $t'$ occur in $u_1\cdots u_{m-1}$, equality in these cases follows from $P(p,m-1)$.

Thus all factors in the two instances of \eqref{frcocycle} agree except possibly the $i=p$ factors.  Since both products equal $1$, the $i=p$ factors also agree.  Therefore
\begin{equation*}
        f(a_1,\ldots,a_{p-1},x,a_{p+1},\ldots,a_n)
        =
        f(a_1,\ldots,a_{p-1},x',a_{p+1},\ldots,a_n),
\end{equation*}
which proves $P(p,m)$.  This completes the double induction.

It follows that $\overline f$ is a group cocycle. Indeed, for $g_1,\ldots,g_{n+1}\in U(R)$ choose $x_r\in I_{g_r}$ and $y_i\le x_i x_{i+1}$; then $\deg y_i=g_i g_{i+1}$, so \eqref{frcocycle} becomes exactly the group-cocycle identity.
\end{proof}

\subsection{Categorical \texorpdfstring{$2$- and $3$-cocycles}{2- and 3-cocycles}}

Let $R$ be a based ring with basis $I$ and structure constants $N_{ij}^k$. Recall that a {\it categorification} of $R$ over $\k$ is a semisimple tensor category $\C$ with an identification ${\rm Gr}(\C)\cong R$; i.e., the simple objects of $\C$ are $X_i$, $i\in I$, with $X_i\otimes X_j\cong \oplus_k N_{ij}^k X_k$ (\cite{EGNO}, Subsection 4.10).

\begin{definition} Let $\C$ be a categorification of $R$ with associativity isomorphism $\alpha$. The definitions in degrees two and three are relative to this fixed categorification. We use normalized cochains when the unit constraints are fixed.

(i) A function $f:I^2\to \k^\times$ is called a {\it categorical $2$-cocycle} for $\C$ if the morphisms
$f(i,j)\id_{X_iX_j}$ extended additively define a tensor structure on the identity functor of $\C$.

(ii) A function $f:I^3\to \k^\times$ is called a {\it categorical $3$-cocycle} for $\C$ if the morphisms
$f(i,j,k)\alpha_{X_iX_jX_k}$ extended additively define a new associativity constraint on $\C$.
\end{definition}

\begin{example}\label{categorical-cocycle-examples} (i) The pullback to $I$ of a group 2-cocycle on $U_\C$ is a categorical 2-cocycle for $\C$. The pullback to $I$ of a group $3$-cocycle on $U_\C$ is a categorical 3-cocycle for $\C$.

(ii) Let $\D$ be a tensor category and $F: \C\to \D$ a tensor functor with tensor structure $\mathcal J$. Say that a collection of $h_i\in\Aut F(X_i)$, $i\in I$, with $h_1=\id$ defines a {\it
projective tensor automorphism} of $F$ if there exist scalars $\beta(i,j)\in \k^\times,\ i,j\in I$ such that
$$
\mathcal J_{X_iX_j}(h_i\otimes h_j)\mathcal J_{X_iX_j}^{-1}|_{F(X_k)}=\beta(i,j)h_k
$$
for all $i,j\in I$ and $X_k\hookrightarrow X_i\otimes X_j$. Extend $h$ by naturality to all objects. The transported tensor structure $h_{X_i\otimes X_j}^{-1}\mathcal J_{ij}(h_i\otimes h_j)$ equals $\beta(i,j)\mathcal J_{ij}$. Comparing its coherence equation with that of $\mathcal J$ and using faithfulness of $F$ gives the categorical $2$-cocycle condition for $\beta$.

(iii) Let $\D$ be a tensor category and $F:\C\to\D$ a nonzero $\k$-linear additive functor with a specified unit isomorphism $\mathbf{1}_\D\cong F(\mathbf{1}_\C)$.
Say that a collection of isomorphisms
$$
\mathcal J_{ij}: F(X_i)\otimes F(X_j)\to F(X_i\otimes X_j)
$$
defines a {\it projective tensor structure} on $F$ if it is compatible with the unit and satisfies the tensor-functor axiom after additive extension, with $\alpha_{i,j,k}$ replaced on simple triples by $\gamma(i,j,k)\alpha_{i,j,k}$ for some normalized function $\gamma:I^3\to\k^\times$. Indeed, $F(X_i)\ne0$ for every $i$: otherwise $F(X_i\otimes X_i^*)=0$, although it contains $F(\mathbf{1})$ as a summand. Semisimplicity and $\k$-linearity now imply faithfulness on morphisms. Transporting the target pentagon by the tensor structures proves the image under $F$ of the pentagon for $\gamma\alpha$; faithfulness gives the pentagon in $\C$. Thus $\gamma$ is a categorical $3$-cocycle.
\end{example}

\begin{example}\label{heisenberg-categorical-cocycle}
The construction of Example~\ref{categorical-cocycle-examples}(ii) can yield a nontrivial categorical $2$-cocycle. Indeed, let $p$ be prime and let
\[
\begin{gathered}
K=\{x(a,b,c)\mid a,b,c\in\mathbb F_p\},\\
x(a,b,c)x(a',b',c')=x(a+a',b+b',c+c'+ab')
\end{gathered}
\]
be the finite Heisenberg group. Put $A=K/Z(K)\cong(\mathbb Z/p)^2$ and let
$H=\operatorname{Fun}(K)\#\mathbb C A$, where $A$ acts on $K$ by conjugation. The quotient $K\to A$ on the group-algebra factor gives a Hopf surjection $D(K)=\operatorname{Fun}(K)\#\mathbb CK\to H$, since conjugation factors through $A$. Thus $H$ inherits the quasitriangular structure of the double. It is semisimple, and for the quotient homomorphism $q:K\to A$, the subalgebra $q^*\operatorname{Fun}(A)$ is the largest central Hopf subalgebra of the form $\operatorname{Fun}(G)$. Indeed, on a central Hopf subalgebra the $R$-matrix commutes with the coproduct, so its coproduct is cocommutative and $G$ is abelian. A grouplike of the smash product has the form $\chi\#a$. Commuting with all functions on $K$ forces $a=1$, since $A$ acts faithfully; every character $\chi$ is conjugation-invariant. The central grouplikes are therefore $\Hom(K,\mathbb C^\times)=\Hom(A,\mathbb C^\times)$, whose span is $q^*\operatorname{Fun}(A)$. The identification of the maximal central Hopf subalgebra with $\operatorname{Fun}(U_{\Rep(H)})$ now gives $U_{\Rep(H)}=A$ by \cite[Theorem~3.8]{GN}.

Fix a primitive $p$-th root of unity $\zeta$ and set
$\psi((a,b),(a',b')):=\zeta^{ab'}$. This cocycle is not symmetric, hence $[\psi]\ne0$ in $H^2(A,\mathbb C^\times)$. Nevertheless, if $u\in\operatorname{Fun}(K)^\times$ is given by $u(x(a,b,c))=\zeta^{-c}$, then
\[
J_\psi=\Delta(u)^{-1}(u\otimes u)
\]
in $H\otimes H$. Conjugation changes the $c$-coordinate, so $u$ is not central. By Lemma~\ref{inje}, the class of $J_\psi$ in $\Aut_1(\Rep(H))$ is nontrivial. Thus no central gauge transformation can trivialize it (Remark~\ref{softaut}). Let $F:\Rep(H)\to\Vec$ be the fiber functor and let $u_i$ be the action of $u$ on $F(X_i)$. If $g_i=\deg X_i\in A$, then $u\otimes u=\Delta(u)J_\psi$ implies that, for $X_k\hookrightarrow X_i\otimes X_j$,
\[
(u_i\otimes u_j)|_{F(X_k)}=\psi(g_i,g_j)u_k.
\]
Thus $u$ is a projective tensor automorphism of $F$, with the nontrivial categorical $2$-cocycle $\beta(i,j)=\psi(g_i,g_j)$.
\end{example}

\begin{example}
The construction of Example~\ref{categorical-cocycle-examples}(iii) can also yield a nontrivial categorical $3$-cocycle, already for a cocommutative Hopf algebra. Let $K=\mathbb Z/4$, let $A=\mathbb Z/2$ act on $K$ by inversion, and set
\[
H=\operatorname{Fun}(K)\#\mathbb C A\cong \mathbb C D_8,
\]
where $D_8$ is the dihedral group of order $8$. For the quotient map $q:K\to A$, the subalgebra $q^*\operatorname{Fun}(A)$ is the maximal central Hopf subalgebra of $H$: under the displayed identification it corresponds to $\mathbb C Z(D_8)$, and $Z(D_8)\cong\mathbb Z/2$. Thus $U_{\Rep(H)}=A$.

Write every $x\in K$ uniquely as $x=q(x)+2t(x)$, with $q(x),t(x)\in\{0,1\}$. Define an invertible element $J\in\operatorname{Fun}(K)^{\otimes 2}\subset H^{\otimes 2}$ by
\[
J(x,y)=(-1)^{q(x)t(y)}.
\]
Since
\[
t(y+z)=t(y)+t(z)+q(y)q(z)\pmod 2,
\]
one has
\[
\frac{J(y,z)J(x,y+z)}{J(x+y,z)J(x,y)}
=(-1)^{q(x)q(y)q(z)}.
\]
Equivalently,
\[
(\id\otimes\Delta)(J)(1\otimes J)
=(q^*\omega)(\Delta\otimes\id)(J)(J\otimes1),
\qquad
\omega(a,b,c)=(-1)^{abc}.
\]
Now let $F:\Rep(H)\to\Vec$ be the fiber functor, and let $\mathcal J_{ij}$ be the action of $J$ on $F(X_i)\otimes F(X_j)$. If $g_i=\deg X_i\in A$, the two sides of the tensor structure axiom differ by
\[
\gamma(i,j,k)=\omega(g_i,g_j,g_k)=(-1)^{g_i g_j g_k}.
\]
Thus $\mathcal J$ is a projective tensor structure on $F$ whose defect is a nontrivial categorical $3$-cocycle. Indeed, $\omega(1,1,1)=-1$, whereas every normalized $3$-coboundary on $\mathbb Z/2$ takes the value $1$ at $(1,1,1)$. Thus $[\omega]$ is the nonzero element of $H^3(\mathbb Z/2,\mathbb C^\times)\cong\mathbb Z/2$, although its inflation to $H^3(\mathbb Z/4,\mathbb C^\times)$ is zero.
\end{example}

\subsection{Categorical \texorpdfstring{$4$-cocycles}{4-cocycles}}
\begin{definition} A {\it projective categorification} of $R$ is a $\k$-linear semisimple category $\C$ with simple objects $X_i,i\in I$, a tensor product functor with Grothendieck ring $R$, and an associativity isomorphism $\alpha$ satisfying pentagon up to a scalar factor $\delta(i,j,k,l)\in \k^\times$ on each tensor product $((X_i\otimes X_j)\otimes X_k)\otimes X_l$, such that $\C$ is rigid. We use a normalized unit convention: $X_1=\mathbf{1}$ has specified natural left and right unit identifications, and each associator with a unit input is the canonical map under those identifications. In particular $\delta=1$ when any input is the unit. Rigidity is specified below.  A function $\delta: I^4\to \k^\times$ arising in this way is called a {\it categorical $4$-cocycle} for $R$.
\end{definition}
The defect scalar is unique because products of simples in a based ring are nonzero.\footnote{More generally, weak rigidity gives, for each simple $Y$, a simple $Y'$ with $\mathbf{1}\subset Y\otimes Y'$. Then $(X\otimes Y)\otimes Y'\cong X\otimes(Y\otimes Y')$ contains $X$. Without weak rigidity, take simples $X_n$ indexed by $\mathbb Z$, with unit $X_0$ and $X_m\otimes X_n=0$ for $m,n\ne0$. The canonical associator satisfies the pentagon, but a pentagon on the zero object imposes no condition on its scalar.} In degree four the associator itself is part of the data. For $R=\mathbb Z[G]$ every normalized associator is a scalar $3$-cochain on $G$ in a choice of product bases, and its defect is its group coboundary. Thus the functions arising in this example are exactly the normalized $4$-coboundaries.

Here by rigidity we mean that every simple object $X=X_i$ is equipped with coevaluation and evaluation maps
$$
\coev_X: \mathbf{1}\to X\otimes X^*,\
\ev_X: X^*\otimes X\to \mathbf{1},
$$
where $X^*:=X_{i^*}$, such that the composition
$$
X\to (X\otimes X^*)\otimes X\to X\otimes (X^*\otimes X)\to X
$$
is the identity.

\enlargethispage{2pt}
Write the two contraction maps for $f:Y\to X\otimes Z$, $g:X^*\otimes Y\to Z$:
\[
 \begin{aligned}
 L(f)&=(\ev_X\otimes1)\alpha^{-1}_{X^*,X,Z}(1\otimes f),\\
 R(g)&=(1\otimes g)\alpha_{X,X^*,Y}(\coev_X\otimes1).
 \end{aligned}
\]
For simple $Z$, naturality moves $f$ to the right in $RL(f)$. The remaining rebracketing contains
\[
\begin{aligned}
 &(1\otimes\alpha^{-1}_{X^*,X,Z})\alpha_{X,X^*,X\otimes Z}\\
 &\quad=\delta(i,i^*,i,z)^{-1}\,
 \alpha_{X,X^*\otimes X,Z}(\alpha_{X,X^*,X}\otimes1)
 \alpha^{-1}_{X\otimes X^*,X,Z}.
\end{aligned}
\]
Composing with coevaluation and evaluation and using the first zigzag gives $RL(f)=\delta(i,i^*,i,z)^{-1} f$, where $Z=X_z$. Taking $Y=\mathbf{1}$, $Z=X^*$ and $f=\coev_X$ shows that $L(\coev_X)$, the second zigzag on the simple $X^*$, equals $\delta(i,i^*,i,i^*)^{-1}\id_{X^*}$. For simple $Y,Z$, the based-ring equality $N_{i z}^{y}=N_{i^* y}^{z}$ gives equality of the dimensions of the two Hom spaces. Thus $L$ is an isomorphism with inverse $\delta(i,i^*,i,z) R$. Direct sums give Frobenius reciprocity in general; the indicated scalar correction is applied separately on each simple summand of $Z$.

If $\delta$ depends only on the degrees of its arguments, i.e., is the pullback of $\overline{\delta}: U(R)^4\to \k^\times$ (which we will show to be always true), then $\overline\delta$ is a group 4-cocycle. In this case the category $\C$ is called a {\it $U(R)$-quasi-monoidal category} and $\overline\delta$ is called a {\it pentagonicity defect cocycle}, \cite{JOY}, Section 4.

\begin{example} Let $\D$ be a fusion category, $G$ be a finite group, and let $c: G\to {\rm BrPic}(\D)$ be a
homomorphism (\cite{ENO}). Recall that to lift this homomorphism to
a $G$-extension of $\D$, we first need to consider the first obstruction
$O_3(c)\in H^3(G,{\rm Inv}\mathcal Z(\D))$ where  $\mathcal Z(\D)$ is the Drinfeld center of $\mathcal D$
and Inv denotes the group of invertible objects. If this obstruction vanishes, we get to choose an element
$M$ in a certain torsor over $H^2(G,{\rm Inv}\mathcal Z(\D))$, which pins down the Grothendieck ring $R$ of
the putative $G$-extension $\C$ of $\D$, and we obtain the second obstruction $O_4(c,M)\in H^4(G,\k^\times)$ to existence of an associativity isomorphism of $\C$ satisfying pentagon. In fact, as shown in \cite{ENO} (see also \cite[Section~2.5 and Corollary~8.7]{DN}), such an associativity isomorphism almost exists: there is always one if we slightly relax the pentagon relation to hold only up to a scalar on each tensor product $((X_i\otimes X_j)\otimes X_k)\otimes X_l$,
and this scalar $\delta(i,j,k,l)$ is pulled back from a 4-cocycle representing $O_4(c,M)$; i.e., we don't necessarily have an honest categorification of $R$, but we always have a projective one.
To check rigidity on a simple object of degree $g$, restrict to $H=\langle g\rangle$. Since $H^4(H,\k^\times)=\k^\times/(\k^\times)^{|H|}=1$, a scalar $3$-cochain on $H$ makes the associator satisfy pentagon. The resulting $H$-extension is rigid by \cite[Theorem~8.5]{DN}, and rescaling the associator preserves nonvanishing of the duality zigzags. Thus such a cocycle $\delta$ is an example of a categorical 4-cocycle.
\end{example}

\subsection{A block nonvanishing lemma}

Let $\C$ be a projective categorification of $R$ with simple objects
$X_i,i\in I$ and associativity isomorphism $\alpha$. Let
\begin{equation*}
\begin{aligned}
P^+_{i,j,k,l}
&:=({\rm id}_{X_i}\otimes \alpha_{X_j,X_k,X_l})
   \alpha_{X_i,X_j\otimes X_k,X_l}
   (\alpha_{X_i,X_j,X_k}\otimes{\rm id}_{X_l}),\\
P^-_{i,j,k,l}
&:=\alpha_{X_i,X_j,X_k\otimes X_l}\alpha_{X_i\otimes X_j,X_k,X_l}.
\end{aligned}
\end{equation*}
Thus \(P^+_{i,j,k,l}\) and \(P^-_{i,j,k,l}\) are morphisms from
\(((X_i\otimes X_j)\otimes X_k)\otimes X_l\) to
\(X_i\otimes(X_j\otimes(X_k\otimes X_l))\) and
\begin{equation}
\label{eq:projective-pentagon-convention}
        P^+_{i,j,k,l}=\delta(i,j,k,l)P^-_{i,j,k,l},\
         \delta(i,j,k,l)\in \k^\times .
\end{equation}

For \(M\in \C\), write $M_i:=\Hom(X_i,M)$,
 then $M\simeq \bigoplus_{l\in I} M_l\otimes X_l$.
Define an automorphism $\Omega^{X_i,X_j,X_k}_M: M\longrightarrow M$
by letting it act on the summand \(M_l\otimes X_l\) as multiplication by
\(\delta(i,j,k,l)\).

\begin{lemma}
\label{lem:projective-sliding}
Let $B$ be simple with the duality data specified above, and let $Z\in\C$. Define the composite morphism
\begin{equation*}
\begin{aligned}
        e_{B,Z}:&
        B^*\otimes(B\otimes Z)
        \xrightarrow{\alpha^{-1}_{B^*,B,Z}}
        (B^*\otimes B)\otimes Z
        \xrightarrow{\operatorname{ev}_B\otimes{\rm id}_Z}
        {\bf 1}\otimes Z
        \xrightarrow{\ell_Z}
        Z .
\end{aligned}
\end{equation*}
Then, for any $M\in \C$ and simple $C\in \C$ we have
\begin{equation}
\label{eq:sliding-left}e_{B,C\otimes M}
  ({\rm id}_{B^*}\otimes\alpha_{B,C,M})=
({\rm id}_C\otimes\Omega^{B^*,B,C}_M)
\bigl(e_{B,C}\otimes{\rm id}_M\bigr)
\alpha^{-1}_{B^*,B\otimes C,M}
\end{equation}
as morphisms
$B^*\otimes((B\otimes C)\otimes M)\longrightarrow C\otimes M.$
In particular,
\begin{equation}
\label{eq:sliding-left1}
e_{X_i,X_j\otimes X_d}
  ({\rm id}_{X_i^*}\otimes\alpha_{X_i,X_j,X_d})=
\delta(i^*,i,j,d)
\bigl(e_{X_i,X_j}\otimes{\rm id}_{X_d}\bigr)
\alpha^{-1}_{X_i^*,X_i\otimes X_j,X_d}.
\end{equation}
\end{lemma}

\begin{proof}
It suffices to prove \eqref{eq:sliding-left1}. The projective pentagon
\eqref{eq:projective-pentagon-convention} for the four simple objects
\(X_i^*,X_i,X_j,X_d\) gives
\begin{equation*}
\begin{aligned}
&({\rm id}_{X_i^*}\otimes \alpha_{X_i,X_j,X_d})
 \alpha_{X_i^*,X_i\otimes X_j,X_d}
 (\alpha_{X_i^*,X_i,X_j}\otimes{\rm id}_{X_d})\\
&\quad =
\delta(i^*,i,j,d)
\alpha_{X_i^*,X_i,X_j\otimes X_d}
\alpha_{X_i^*\otimes X_i,X_j,X_d}.
\end{aligned}
\end{equation*}
Solve this equation for
\({\rm id}_{X_i^*}\otimes\alpha_{X_i,X_j,X_d}\), and then compose on the left
with the evaluation map \(e_{X_i,X_j\otimes X_d}\).  The factor
\(\alpha_{X_i^*,X_i,X_j\otimes X_d}\) cancels with the inverse associator in the
definition of \(e_{X_i,X_j\otimes X_d}\).  The remaining associator
\(\alpha_{X_i^*\otimes X_i,X_j,X_d}\) is removed by naturality with respect to the morphism
\(\operatorname{ev}_{X_i}:X_i^*\otimes X_i\to{\bf 1}\)
 and by the unit constraint.
This gives \eqref{eq:sliding-left1}.
\end{proof}

Now put $H^r_{ij}:=\Hom(X_r,X_i\otimes X_j)$,
        $J_{ij}:=\{r\in I\mid H^r_{ij}\ne 0\}$.
Let
$
        \iota^r_{ij}:H^r_{ij}\otimes X_r\to X_i\otimes X_j,
        \pi^r_{ij}:X_i\otimes X_j\to H^r_{ij}\otimes X_r
$
be the canonical splitting maps. For \(r\in J_{ij}\), define
\begin{equation}
\label{eq:uijr}
        u^r_{ij}:=
        e_{X_i,X_j}({\rm id}_{X_i^*}\otimes\iota^r_{ij}):
        X_i^*\otimes(H^r_{ij}\otimes X_r)\to X_j .
\end{equation}
This is $L(\iota^r_{ij})$ for the contraction map above; hence it is
nonzero, and therefore an epimorphism.

Define the rebracketing isomorphisms
\begin{equation}
\label{eq:rho-fixed-final}
\begin{aligned}
\rho_{i,j,k,l}:&
        (X_i\otimes X_j)\otimes(X_k\otimes X_l)\\
&\xrightarrow{\alpha_{X_i,X_j,X_k\otimes X_l}}
        X_i\otimes(X_j\otimes(X_k\otimes X_l))\\
&\xrightarrow{{\rm id}_{X_i}\otimes\alpha^{-1}_{X_j,X_k,X_l}}
        X_i\otimes((X_j\otimes X_k)\otimes X_l).
\end{aligned}
\end{equation}
and
\begin{equation}
\label{eq:sigma-fixed-final}
        \sigma_{i,j,k,l,m}:=
        ({\rm id}_{X_i}\otimes\rho^{-1}_{j,k,l,m})
        \alpha_{X_i,X_j,(X_k\otimes X_l)\otimes X_m}.
\end{equation}
Thus
\begin{equation*}
  \sigma_{i,j,k,l,m}:      (X_i\otimes X_j)\otimes((X_k\otimes X_l)\otimes X_m)
        \longrightarrow
        X_i\otimes((X_j\otimes X_k)\otimes(X_l\otimes X_m)).
\end{equation*}

\begin{lemma}
\label{lem:block-nonvanishing-final}
(i) If \(a\in J_{ij}\) and \(b\in J_{jk}\), then the block
\begin{equation*}
        Q^{a,b}_{i,j,k}:
        (H^a_{ij}\otimes X_a)\otimes X_k
        \longrightarrow
        X_i\otimes(H^b_{jk}\otimes X_b)
\end{equation*}
of \(\alpha_{X_i,X_j,X_k}\) is nonzero.  The same holds for the inverse
associator.

(ii) If \(a\in J_{ij}\), \(b\in J_{jk}\), \(c\in J_{kl}\), then the block
\begin{equation*}
        R_{i,j,k,l}^{a,b,c}:
        (H^a_{ij}\otimes X_a)\otimes(H^c_{kl}\otimes X_c)
        \longrightarrow
        X_i\otimes((H^b_{jk}\otimes X_b)\otimes X_l)
\end{equation*}
of \(\rho_{i,j,k,l}\) is nonzero.  The same holds for the inverse map
\(\rho^{-1}_{i,j,k,l}\).

(iii) If $a\in J_{ij},b\in J_{jk},
        c\in J_{kl},d\in J_{lm}$
then the block
\begin{equation*}
\begin{aligned}
 S_{i,j,k,l,m}^{a,b,c,d}:\;&
 (H^a_{ij}\otimes X_a)\otimes((H^c_{kl}\otimes X_c)\otimes X_m)\\
 &\longrightarrow X_i\otimes((H^b_{jk}\otimes X_b)\otimes(H^d_{lm}\otimes X_d))
\end{aligned}
\end{equation*}
of \(\sigma_{i,j,k,l,m}\) is nonzero.
\end{lemma}

For $(X,X^*,X)$ the block in part~(i) with both intermediate constituents equal to the unit is the duality zigzag, up to a nonzero scalar. Thus this particular block already requires the nonvanishing in the definition of rigidity.

\begin{proof}
The blocks in (i)--(iii) retain all output simples and their entire multiplicity spaces. Individual entries at a prescribed output simple may vanish. For (i), contract $Q^{a,b}_{i,j,k}$ with evaluation by composing
\[
\begin{aligned}
 X_i^*\otimes((H^a_{ij}\otimes X_a)\otimes X_k)
 &\xrightarrow{{\rm id}\otimes Q^{a,b}_{i,j,k}}
 X_i^*\otimes(X_i\otimes(H^b_{jk}\otimes X_b))\\
 &\xrightarrow{e_{X_i,H^b_{jk}\otimes X_b}}H^b_{jk}\otimes X_b .
\end{aligned}
\]
By Lemma~\ref{lem:projective-sliding} with
 \(M=X_k\), this contraction is
\(\delta(i^*,i,j,k)\) times the composite \scriptsize
\begin{equation*}
        X_i^*\otimes((H^a_{ij}\otimes X_a)\otimes X_k)
        \xrightarrow{\alpha^{-1}}
        (X_i^*\otimes(H^a_{ij}\otimes X_a))\otimes X_k
        \xrightarrow{u^a_{ij}\otimes{\rm id}_{X_k}}
        X_j\otimes X_k
        \xrightarrow{\pi^b_{jk}}
        H^b_{jk}\otimes X_b .
\end{equation*} \normalsize
The middle arrow is
an epimorphism and \(\pi^b_{jk}\ne 0\), so this composite is nonzero. It follows that
\(Q^{a,b}_{i,j,k}\ne 0\).  Apply the result just proved in $\C^{\rm op}$, with the same tensor product on objects and associator $(\alpha^{-1})^{\rm op}$. The defect there is $\delta^{-1}$. To obtain the duality data for $X$ use the reversed evaluation and coevaluation of $X^*$: the first zigzag in the opposite category is the second zigzag for $X^*$ in $\C$, which is nonzero by the calculation above. Rescale one map to normalize it. Splittings in $\C^{\rm op}$ reverse inclusions and projections. Thus the forward block in $\C^{\rm op}$ is the opposite of the desired inverse block in $\C$.

For (ii), compose
\begin{equation*}
        X_i^*\otimes\bigl((H^a_{ij}\otimes X_a)
        \otimes(H^c_{kl}\otimes X_c)\bigr)
        \xrightarrow{{\rm id}\otimes R_{i,j,k,l}^{a,b,c}}
        X_i^*\otimes\bigl(X_i\otimes((H^b_{jk}\otimes X_b)\otimes X_l)\bigr)
\end{equation*}
with \(e_{X_i,(H^b_{jk}\otimes X_b)\otimes X_l}\).
By Lemma~\ref{lem:projective-sliding} with
\(M=H^c_{kl}\otimes X_c\), this contraction is
\(\delta(i^*,i,j,c)\) times the composite \scriptsize
\begin{equation}
\label{eq:rho-trace-comparison}
        X_i^*\otimes\bigl((H^a_{ij}\otimes X_a)
        \otimes(H^c_{kl}\otimes X_c)\bigr)
        \longrightarrow
        X_j\otimes(H^c_{kl}\otimes X_c)
        \xrightarrow{\overline Q^{c,b}_{j,k,l}}
        (H^b_{jk}\otimes X_b)\otimes X_l .
\end{equation} \normalsize
Here the first arrow is obtained from \(u^a_{ij}\) by precomposing with an
associator, and \(\overline Q^{c,b}_{j,k,l}\) is the block
of \(\alpha^{-1}_{X_j,X_k,X_l}\) from
\(X_j\otimes(H^c_{kl}\otimes X_c)\) to
\((H^b_{jk}\otimes X_b)\otimes X_l\).  This block is nonzero by the
inverse-associator part of (i).  Since the first arrow in
\eqref{eq:rho-trace-comparison} is an epimorphism, the composite is nonzero.
Thus the contraction of \(R_{i,j,k,l}^{a,b,c}\) is nonzero, and hence
\(R_{i,j,k,l}^{a,b,c}\ne 0\).  For the inverse block apply (ii), up to this point, in $\C^{\rm op}$. By \eqref{eq:rho-fixed-final}, $\rho_{\C^{\rm op}}=(\rho_{\C}^{-1})^{\rm op}$: reversing arrows reverses the order of its two constituent maps and inverts each associator. The source and target projectors are interchanged, so this gives every stated inverse block.

For (iii), set
\begin{equation*}
        M:=(H^c_{kl}\otimes X_c)\otimes X_m.
\end{equation*}
Write $A=H^a_{ij}\otimes X_a$. The contraction of \(S_{i,j,k,l,m}^{a,b,c,d}\) is a morphism
\begin{equation*}
        X_i^*\otimes\bigl((H^a_{ij}\otimes X_a)\otimes M\bigr)
        \longrightarrow
        (H^b_{jk}\otimes X_b)\otimes(H^d_{lm}\otimes X_d).
\end{equation*}
By Lemma~\ref{lem:projective-sliding} the contraction is
\begin{equation}
\label{eq:sigma-trace-comparison}
        \overline R^{c}_{b,d}
        ({\rm id}_{X_j}\otimes \Omega^{X_i^*,X_i,X_j}_{M})
        (u^a_{ij}\otimes{\rm id}_{M})
        \alpha^{-1}_{X_i^*,A,M}.
\end{equation}
Indeed, naturality of $e_{X_i,-}$ moves the target projection and the map $\rho^{-1}_{j,k,l,m}$ past the evaluation. Applying \eqref{eq:sliding-left} to the remaining associator gives the last three factors of \eqref{eq:sigma-trace-comparison}. Here
\(\overline R^{c}_{b,d}\) is the block of
\(\rho^{-1}_{j,k,l,m}\) from $X_j\otimes((H^c_{kl}\otimes X_c)\otimes X_m)$
to $(H^b_{jk}\otimes X_b)\otimes(H^d_{lm}\otimes X_d)$,
which is nonzero by the inverse part of (ii).  The map
\(u^a_{ij}\otimes{\rm id}_{M}\) is an epimorphism, and
\(\Omega^{X_i^*,X_i,X_j}_{M}\) is an automorphism.
  Therefore the composite
\eqref{eq:sigma-trace-comparison} is nonzero, and hence
\(S_{i,j,k,l,m}^{a,b,c,d}\ne 0\).
\end{proof}

\subsection{Classification of categorical cocycles}

The nonvanishing lemma allows us to apply the based-ring cocycle criterion to each categorical cocycle.
\begin{theorem}\label{catcoc} In each of the cases \textup{(i)} $n=2$, \textup{(ii)} $n=3$, and \textup{(iii)} $n=4$, every categorical $n$-cocycle is the pullback of a unique group $n$-cocycle on $U(R)$.
\end{theorem}

\begin{proof} By Proposition \ref{anydegreecocycle}, in all three cases it suffices to show that
the corresponding categorical cocycle is a based-ring cocycle.

(i) Let $f:I^2\to \k^\times$ be a categorical $2$-cocycle. The tensor structure coherence equation for the identity functor gives
\begin{equation*}
        f(i,j)f(a,k)=f(i,b)f(j,k)
\end{equation*}
for $a\in J_{ij}$ and $b\in J_{jk}$ as long as $Q_{i,j,k}^{a,b}\ne 0$. Thus by Lemma~\ref{lem:block-nonvanishing-final}(i), this equality holds for all such $a$ and $b$, as desired.

(ii) Let $f:I^3\to \k^\times$ be a categorical $3$-cocycle. Comparing the pentagon for the associator multiplied by $f$ with the original pentagon gives
\begin{equation*}
        f(a,k,l)f(i,j,c)
        =
        f(i,j,k)f(i,b,l)f(j,k,l)
\end{equation*}
for $a\in J_{ij},\  b\in J_{jk},\  c\in J_{kl}$ as long as $R_{i,j,k,l}^{a,b,c}\ne 0$. Thus by Lemma~\ref{lem:block-nonvanishing-final}(ii), this equality holds for all such $a,b,c$, as desired.

(iii) Consider five
simple objects $X_i,X_j,X_k,X_l,X_m$ and choose
\begin{equation*}
        a\in J_{ij},\  b\in J_{jk},\
        c\in J_{kl},\  d\in J_{lm}.
\end{equation*}
Let $S_{i,j,k,l,m}^{a,b,c,d}$ be the nonzero block of $\sigma_{i,j,k,l,m}$ from Lemma
\ref{lem:block-nonvanishing-final}(iii).

Write the five inputs as $1,2,3,4,5$ in order (these digits denote positions). Use the following names for their bracketings, omitting the outermost parentheses:
\[
\begin{array}{c@{\;}l@{\qquad}c@{\;}l}
 A&1(2(3(45))) & H&(1(23))(45)\\
 B&1(2((34)5)) & I&((12)3)(45)\\
 C&1((23)(45)) & J&(1(2(34)))5\\
 D&1((2(34))5) & K&(1((23)4))5\\
 E&1(((23)4)5) & L&((12)(34))5\\
 F&(12)(3(45)) & M&((1(23))4)5\\
 G&(12)((34)5) & N&(((12)3)4)5 .
\end{array}
\]
A path such as $NLG$ denotes the composite of the indicated forward associators, read from left to right. Products of maps are composed from right to left. Put $P=NLG$, $Q=CA$, $T=NMHC$. Then $\sigma:G\to C$ of \eqref{eq:sigma-fixed-final} satisfies $Q\sigma=GBA=GFA$; the last equality is naturality.

For $r=1,2,3,4$, let $D_r$ multiply the summand $x$ of the adjacent pair $(r,r+1)$ by the four-input defect with that pair replaced by $x$ and the three other inputs left in order. Take $D_1,D_3$ at $G$ and $D_2,D_4$ at $C$. The two operators at each vertex commute because the adjacent pairs are disjoint. Within a face that keeps a given pair bracketed, every edge intertwines its summand projections by naturality. This allows its $D_r$ to be evaluated at any vertex of that face, before or after following an edge.

We use the convention $P^+=\delta P^-$ of \eqref{eq:projective-pentagon-convention}. In the face with pair $(12)$, $NLGF$ is the long route and $NIF$ the short one. Thus $Q\sigma D_1^{-1}P=NIFA$. In the face with pair $(45)$, $IHCA$ is the long route and $IFA$ the short one. The square $NIH=NMH$ then gives
\begin{equation}\label{eq:five-lower}
 Q D_4\sigma D_1^{-1}P=NIHCA=NMHCA=QT.
\end{equation}
For the remaining faces put $s=\delta(i,j,k,l)$ and $t=\delta(j,k,l,m)$. Starting with $Q\sigma P= NLGBA$, insert $D_3$ at $G$. The successive path identities are
\begin{equation}\label{eq:five-upper}
\begin{aligned}
 Q\sigma D_3P
   &= NLJDBA &&(LJDB\text{ versus }LGB)\\
   &= s^{-1}NMKJDBA &&(NMKJ\text{ versus }NLJ)\\
   &= s^{-1}NMKEDBA &&(KJD=KED)\\
   &= s^{-1}t\,NMKECA &&(EDBA\text{ versus }ECA)\\
   &= s^{-1}t\,Q D_2T &&(MKEC\text{ versus }MHC).
\end{aligned}
\end{equation}
In the four pentagon comparisons the first path is the long route. The middle equality uses the commuting square. Substituting \eqref{eq:five-lower} in \eqref{eq:five-upper} and cancelling $P,Q$ gives the operator identity
\[
 t\,D_2D_4\,\sigma=s\,\sigma D_1D_3 .
\]
Projecting at $G$ to $(a,c)$ and at $C$ to $(b,d)$ gives \scriptsize
\begin{equation}\label{eq:K5-block-relation}
        \delta(j,k,l,m)\delta(i,b,l,m)\delta(i,j,k,d)S_{i,j,k,l,m}^{a,b,c,d}
        =
        \delta(a,k,l,m)\delta(i,j,c,m)\delta(i,j,k,l)S_{i,j,k,l,m}^{a,b,c,d}.
\end{equation} \normalsize
Since $S_{i,j,k,l,m}^{a,b,c,d}$ is nonzero, we may cancel it and obtain
\begin{equation}
\label{eq:Pi-fusion-4-cocycle}
        \delta(j,k,l,m)\delta(i,b,l,m)\delta(i,j,k,d)
        =
        \delta(a,k,l,m)\delta(i,j,c,m)\delta(i,j,k,l).
\end{equation}
This is \eqref{frcocycle} in degree four for independently chosen adjacent constituents. Proposition~\ref{anydegreecocycle} now gives the claimed descent.
\end{proof}

In the rigid case, Theorem~\ref{catcoc}(iii) answers Question~4.3 of \cite{JOY} with $G=U(R)$. This grading is faithful because every element of $U(R)$ is the degree of a simple object. Without rigidity the question has a negative answer, even when all products of simples are nonzero.

\begin{example}\label{nonrigid-defect}
Let $M=\{1,e\}$ be the monoid with $e^2=e$, and consider the semisimple monoidal category $\C$ of finite-dimensional $M$-graded vector spaces. Denote the two simple objects by $\mathbf{1}$ and $E$, so that $E\otimes E\cong E$. In particular $E$ has no dual, since $E\otimes Y$ has no unit summand for either simple object $Y$.

Choose $q\in\k^\times$ with $q\ne1$. Rescale the usual associator by $q$ on $(E,E,E)$ and extend by additivity. On $(E,E,E,E)$ the long and short paths around the pentagon are multiplied by $q^3$ and $q^2$, respectively. Thus the defect is $q$ on this quadruple and $1$ on $(\mathbf{1},E,E,E)$ by unit normalization. In any group grading, $\deg(E)^2=\deg(E)$ implies $\deg(E)=1_G$. Both simples are therefore neutral, so faithfulness forces $G=\{1\}$. The two quadruples have the same degree tuple $(1,1,1,1)$ but different defects, so the rescaled category is not $G$-quasi-monoidal for any group $G$.
\end{example}

\end{document}